\documentclass{amsart}
\usepackage{amsmath,amsthm,latexsym,amssymb,hyperref, amsfonts}
\usepackage{stmaryrd}
\usepackage{eucal}
\usepackage{mathrsfs}
\usepackage[all]{xy}
\usepackage{color}
\usepackage[numbers]{natbib}
\usepackage{url}
\allowdisplaybreaks
\long\def\emptytext#1{}
\numberwithin{equation}{section}

\newtheorem{thm}{Theorem}[section]
\newtheorem{lem}[thm]{Lemma}
\newtheorem{prop}[thm]{Proposition}
\newtheorem{cor}[thm]{Corollary}

\theoremstyle{definition}
\newtheorem{defn}[thm]{Definition}
\theoremstyle{remark}
\newtheorem{rmk}[thm]{Remark}
\newtheorem{ex}[thm]{Example}

\newtheorem{notn}[thm]{Notation}

\newtheorem{convention}[thm]{Convention}

\newcommand\susp{\Sigma^{\infty}_{+}}
\newcommand{\cE}{\mathscr E}
\newcommand\Mod{\cat{Mod}}
\newcommand{\bP}{\mathbb P}
\newcommand{\bZ}{\mathbb Z}
\newcommand{\bF}{\mathbb F}
\newcommand{\bQ}{\mathbb Q}
\newcommand{\bC}{\mathbb C}

\newcommand{\Ret}{\operatorname{Ret}}
\newcommand\Spec{\cat{Sp}}
\newdir{ >}{{}*!/-10pt/@{>}}
\newcommand\fib{\ar @{->>} [r]} 
\newcommand\cof{\ar @{ >->}[r]}

\newcommand \Om{\Omega}

\newcommand \del{\partial}

\newcommand \ve{\varepsilon}

\newcommand\cat{\mathsf}
\newcommand{\ob}{\operatorname{Ob}}

\newcommand{\G}{\mathbb G}

\newcommand{\bS}{\mathbb S}

\newcommand{\bbS}{\mathbb S}
\renewcommand{\P}{\mathbb P}

\newcommand\egal[2]{\overset {#1}{\underset {#2}\rightrightarrows }}
\newcommand\pist{\pi_{*}^{s}}

\renewcommand\P{\mathbb P}
\newcommand\id{\mathrm{Id}}

\newcommand\sgx{\susp \G X}

\newcommand\comod{\cat{Comod}}

\newcommand\cC{\mathcal C}
\newcommand\cD{\mathcal D}

\newcommand\cL{\mathcal L}

\newcommand\colim{\operatorname{colim}}

\newcommand\map{\operatorname{Map}}
\newcommand\Tot{\operatorname{Tot}}
\newcommand{\iso}{\cong}

\newcommand\soxp{\susp \Omega X}

\newcommand\sxp{\susp X}

\newcommand{\sm}{\wedge}

\newcommand{\cH}{\mathcal H}

\newcommand{\THH}{\textnormal{THH}}
\newcommand{\coTHH}{\textnormal{coTHH}}

\newcommand{\tensor}{\otimes}
\renewcommand{\P}{\mathcal P}
\newcommand\si{s^{-1}}
\newcommand\sdr[4]{\xymatrix{#1\ar @<1ex>[rr]^-{#3}&&#2\ar @<1ex>[ll]^-{#4}}}

\newcommand\adjunction[4]{\xymatrix{#1\ar @<1.18ex>[rr]^{#3}&\perp&#2\ar @<1.18ex>[ll]^{#4}}}
\newcommand{\hugeadjunction}[4]{\xymatrix{ #1 \ar@<1.18ex>[rrr]^-{#3} &&& #2 \ar@<1.18ex>[lll]^-{#4}}}
\newcommand{\giantadjunction}[4]{\xymatrix{ #1 \ar@<1.18ex>[rrrr]^-{#3} &&\perp&& #2 \ar@<1.18ex>[llll]^-{#4}}}
\newcommand{\bigadjunction}[4]{\xymatrix{ #1 \ar@<1ex>[rr]^-{#3} && #2 \ar@<1ex>[ll]^-{#4}}}

\newcommand{\cohoch}{\widehat{\mathcal H}}
\newcommand{\hoch}{{\mathcal H}}

\newcommand{\Thick}{\operatorname{Thick}}

\newcommand{\wb}{\overline}
\newcommand{\oL}{\wb L}
\newcommand{\oR}{\wb R}
\newcommand{\oT}{\wb \Tot\, }
\newcommand{\oLdot}{\wb \Ldot}
\newcommand{\Ldot}{L^{\bullet}}
\newcommand{\oRdot}{\wb \Rdot}
\newcommand{\Rdot}{R^{\bullet}}

\newcommand{\con}{c^{\bullet}}

\newcommand{\Cobar}{\textnormal{Cobar}}
\newcommand{\bBar}{\textnormal{Bar}}

\address{EPFL SV BMI UPHESS, Station 8, CH-1015 Lausanne, Switzerland}
\email{kathryn.hess@epfl.ch}

\address{Department of Mathematics, Statistics, and Computer Science, University of Illinois at
Chicago, 508 SEO m/c 249,
851 S. Morgan Street,
Chicago, IL, 60607-7045, USA}
    \email{shipleyb@uic.edu}
    
    \keywords{ Topological Hochschild homology, coalgebras}
\subjclass[2010]{Primary:{ 19D55, 16T15; Secondary: 16E40, 55U35, 55P43 }}

\begin{document}
\title{Invariance properties of coHochschild homology}
\author{Kathryn Hess}
\author{Brooke Shipley}

\date{\today}
\begin{abstract} The notion of Hochschild homology of a dg algebra admits a natural dualization, the coHochschild homology of a dg coalgebra, introduced in \cite{hps} as a tool to study free loop spaces.  In this article we prove ``agreement'' for coHochschild homology, i.e., that the coHochschild homology of a dg coalgebra $C$ is isomorphic to the Hochschild homology of the dg category of appropriately compact $C$-comodules, from which Morita invariance of coHochschild homology follows.  Generalizing the dg case, we define the topological coHochschild homology (coTHH) of coalgebra spectra, of which suspension spectra are the canonical examples, and show that coTHH of the suspension spectrum of a space $X$ is equivalent to the suspension spectrum of the free loop space on $X$, as long as $X$ is a nice enough space {(for example, simply connected.)}   Based on this result and on a Quillen equivalence established in \cite{HS16}, we prove that ``agreement'' holds for coTHH as well.  
\end{abstract}

\maketitle
\tableofcontents
\section{Introduction}
For any commutative ring $\Bbbk$, the classical definition of Hochschild homology of $\Bbbk$-algebras \cite{loday} admits a straightforward extension to differential graded (dg)  $\Bbbk$-algebras.  In \cite{mccarthy} McCarthy extended the definition of Hochschild homology in another direction, to $\Bbbk$-exact categories, seen as $\Bbbk$-algebras with many objects. As Keller showed in \cite{keller:cyclic}, there is a common refinement of these two extended definitions to dg categories, seen as dg algebras with many objects.   This invariant of dg categories satisfies many useful properties, including ``agreement'' (the Hochschild homology of a dg algebra is isomorphic to that of the dg category of compact modules) \cite [2.4] {keller:cyclic} and Morita invariance (a functor in the homotopy category of dg categories that induces an isomorphism between the subcategories of compact objects also induces an isomorphism on Hochschild homology) \cite[4.4]{toen}.

The notion of Hochschild homology of a differential graded (dg) algebra admits a natural dualization, the \emph{coHochschild homology} of a dg coalgebra, which was introduced by Hess, Parent, and Scott in \cite{hps}, generalizing the non-differential notion of \cite{doi}. They showed in particular that the coHochschild homology of the chain coalgebra on a simply connected space $X$ is isomorphic to the homology of the free loop space on $X$ and that the coHochschild homology of a connected dg coalgebra $C$ is isomorphic to the Hochschild homology of $\Om C$, the cobar construction on $C$.  

In this article we establish further properties of coHochschild homology, analogous to the invariance properties of Hochschild homology recalled above.  We first prove a sort of categorification of the relation between coHochschild homology of a connected dg coalgebra $C$ and the Hochschild homology of $\Om C$, showing that there is a dg Quillen equivalence between the categories of $C$-comodules and of $\Om C$-modules (Proposition \ref{prop:cobar}).  We can then establish an ``agreement''-type result, stating that the coHochschild homology of a dg coalgebra $C$ is isomorphic to the Hochschild homology of the dg category spanned by certain compact $C$-comodules (Proposition ~\ref{prop:cohoch-props}).  Thanks to this agreement result, we can show as well that coHochschild homology is a Morita invariant (Proposition ~\ref{prop.MT}), using the notion of Morita equivalence of dg coalgebras formulated in \cite{berglund-hess}, which extends that of Takeuchi \cite{takeuchi} and which we recall here.  Proving these results required us to provide criteria under which a dg Quillen equivalence of dg model categories induces a quasi-equivalence of dg subcategories (Lemma \ref{lem:q-e}); this technical result, which we were unable to find in the literature, may also be useful in other contexts.

The natural analogue of Hochschild homology for spectra, called \emph{topological Hochschild homology (THH)}, has proven to be an important and useful invariant of ring spectra, particularly because of its connection to K-theory via the Dennis trace.  Blumberg and Mandell proved moreover that THH satisfies both ``agreement,'' in the sense that THH of a ring spectrum is equivalent to THH of the spectral category of appropriately compact $R$-modules, and Morita invariance \cite{BM}

We define here an analogue of coHochschild homology for spectra, which we call \emph{topological coHochschild homology (coTHH)}.  We show that coTHH is homotopy invariant, as well as independent of the particular model category of spectra in which one works.    We prove moreover that coTHH of the suspension spectrum $\susp X$ of a connected Kan complex $X$ is equivalent to $\susp \mathcal LX$, the suspension spectrum of the free loop space on $X$, whenever $X$ is \emph{EMSS-good}, i.e,  whenever $\pi_{1}X$ acts nilpotently on the integral homology of the based loop space on $X$  (Theorem~\ref{thm.free.loop}).  

This equivalence was already known for simply connected spaces $X$, by work of Kuhn \cite{Kuhn} and Malkiewicz \cite{caryM}, though they did not use the term coTHH. The extension of the equivalence to EMSS-good spaces is based on new results concerning total complexes of cosimplicial suspension spectra, such as the fact that 
$\oT (\Sigma^{\infty}Y^{\bullet}) \simeq \Sigma^{\infty} \oT Y^{\bullet}$ whenever the homology spectral sequence for a cosimplicial space $Y^{\bullet}$ with coefficients in $\bZ$ strongly converges (Corollary \ref{cor-conv}). We also show that if $X$ is an EMSS-good space, then the Anderson spectral sequence for homology with coefficients in $\bZ$  for the cosimplicial space $ \map(S^1_{\bullet}, X)$ strongly converges to $H_*(\cL X; \bZ)$ (Proposition \ref{prop.anderson}).  

 In \cite{bok-wald}, B\"okstedt and Waldhausen proved that $\THH(\susp \Om X)\simeq \susp \mathcal LX$ for simply connected $X$.  It follows thus from Theorem \ref{thm.free.loop} that if $X$ is simply connected, then $\THH(\susp \Om X)\simeq \coTHH(\susp X)$, analogous to the result for dg coalgebras established in \cite{hps}.   Combining this result with the spectral Quillen equivalence between categories of $\susp \Om X$-modules and of $\susp X$-comodules established in \cite {HS16} and with THH-agreement \cite{BM}, we obtain coTHH-agreement for simply connected Kan complexes $X$:  $\coTHH(\susp X)$ is equivalent to THH of the spectral category of appropriately compact $\susp \Om X$-modules (Corollary \ref{spec.agree}).  

We do not consider Morita invariance for coalgebra spectra in this article, as the duality requirement of the framework in  \cite{berglund-hess} is too strict to allow for  interesting spectral examples. We expect that a meaningful formulation should be possible in the $\infty$-category context.

In parallel with writing this article,  the second author collaborated with Bohmann, Gerhardt, H\o genhaven, and Ziegenhagen on developing computational tools for coHochschild homology, in particular an analogue of the B\"okstedt spectral sequence for topological Hochschild homology constructed by Angeltveit and Rognes~\cite{Angeltveit-Rognes}. For $C$ a coalgebra spectrum, the $E_2$-page of this spectral sequence is the associated graded of the classical coHochschild homology of the homology of $C$ with coefficients in a field $\Bbbk$, and the spectral sequence abuts to the $\Bbbk$-homology of  $\coTHH(C)$. If $C$ is connected and cocommutative, then this is a spectral sequence of coalgebras.  In \cite{BGHSZ} the authors also proved a Hochschild-Kostant-Rosenberg-style theorem for coHochschild homology of cofree cocommutative differential graded coalgebras.

In future work we will construct and study an analogue of the Dennis trace map, with source the K-theory of a dg or spectral coalgebra $C$  and with target its (topological) coHochschild homology.

{{\em Acknowledgements:}
Work on this article began while the authors were in residence at the Mathematical Sciences Research Institute in Berkeley, California, during Spring 2014, partially supported by the National Science Foundation under Grant No. 0932078000. The second author was also supported by NSF grants DMS-1104396, DMS-1406468, and DMS-1811278 during this work. We would also like to thank the University of Illinois at Chicago, the EPFL, and the University of Chicago for their hospitality during research visits enabling us to complete the research presented in this article.  The authors would also like to thank the Isaac Newton Institute for
   Mathematical Sciences, Cambridge, for support and hospitality during
   the Fall 2018 program ``Homotopy harnessing higher structures" where this paper was
   finished. This work was supported by EPSRC grant no EP/K032208/1. }

%%%%%%%%%%%%%%%%%%%%%%%%%%%%%%

\section{CoHochschild homology for chain coalgebras}\label{sec:dg}

In this section we recall from \cite{hps} the coHochschild complex of a chain coalgebra over a field $\Bbbk$, which generalizes the definitions in {\cite{doi}} and in {\cite{idrissi}} and dualizes the usual definition of the Hochschild complex of a chain algebra.  We establish important properties of this construction analogous to those known to hold for Hochschild homology:  ``agreement'' (in the sense of \cite{mccarthy}) and Morita invariance.

\begin{notn}Throughout this section we work over a field $\Bbbk$ and write $\tensor$ to denote the tensor product over $\Bbbk$ and $|v|$ to denote the degree of a homogeneous element $v$ of a graded vector space.
\begin{itemize}
\item We denote the category of (unbounded) graded chain complexes over $\Bbbk$ by $\cat{Ch}_\Bbbk$, the category of augmented, nonnegatively graded chain algebras (dg algebras) over $\Bbbk$ by $\cat {Alg}_{\Bbbk}$, and the category of coaugmented, connected (and hence nonnegatively graded) chain coalgebras (dg coalgebras) by $\cat{Coalg}_{\Bbbk}$.   All of these categories are naturally dg categories, i.e., enriched over $\cat{Ch}_\Bbbk$, with $\cat{Alg}_{\Bbbk}$ and $\cat{Coalg}_{\Bbbk}$ inheriting their enrichments from that of $\cat{Ch}_{\Bbbk}$.
\item We apply the Koszul sign convention for commuting elements  of a graded vector space or for commuting a morphism of graded vector spaces past an element of the source module.  For example,  if $V$ and $W$ are graded algebras and $v\otimes w, v'\otimes w'\in V\otimes W$, then $$(v\otimes w)\cdot (v'\otimes w')=(-1)^{|w|\cdot |v'|}vv'\otimes ww'.$$ Furthermore, if $f:V\to V'$ and $g:W\to W'$ are morphisms of graded vector spaces, then for all $v\otimes w\in V\otimes W$, 
$$(f\otimes g)(v\otimes w)=(-1)^{|g|\cdot |v|} f(v)\otimes g(w).$$
All signs in the formulas below follow from the Koszul rule. It is a matter of straightforward calculation in each case to show that differentials square to zero.
\item The \emph {desuspension} endofunctor $\si$ on the category of graded vector spaces is defined on objects $V=\bigoplus _{i\in \mathbb Z} V_ i$ by
$(\si V)_ i \cong V_ {i+1}$.  Given a homogeneous element $v$ in
$V$, we write $\si v$ for the corresponding element of $\si V$. 
\item Given chain complexes $(V,d)$ and $(W,d)$, the notation
$f:(V,d)\xrightarrow{\simeq}(W,d)$ indicates that $f$ induces an isomorphism in homology. 
In this case we refer to $f$ as a \emph {quasi-isomorphism}.
\item \cite[Section 2.3]{toen} A \emph{quasi-equivalence} of dg categories is a dg functor $F:\cat C \to \cat D$ such that $F_{X,X'}:\hom_{\cat C}(X,X') \to \hom_{\cat D}\big(F(X), F(X')\big)$ is a quasi-isomorphism for all $X,X'\in \ob \cat C$ (i.e., $F$ is \emph{quasi-fully faithful}) and such that the induced functor on the \emph{homology categories}, $H_0F: H_0 \cat C \to H_0\cat D$, is essentially surjective, i.e., $F$ is \emph{quasi-essentially surjective}.  The objects of the homology category $H_0\cat C$, which is a dg category in which the hom-objects have zero differential, are the same as those of $\cat C$, while hom-objects are given by the $0^{\text{th}}$-homology of the hom-objects of $\cat C$. 
\item Let $T$ denote the endofunctor on the category of graded vector spaces given by
$$TV=\oplus _{n\geq 0}V^{\otimes n},$$
where $V^{\otimes 0}= \Bbbk$.  An element of the summand $V^{\otimes n}$ of $TV$ is denoted $v_{1}|\cdots |v_{n}$, where $v_{i}\in V$ for all $i$. 
\item The coaugmentation coideal of any $C$ in $\cat {Coalg}_{\Bbbk}$ is denoted $\overline C$.
\item We consistently apply the Einstein summation convention, according to which an expression involving a term with the same letter as a subscript and a superscript denotes a sum over that index, e.g., $c_{i}\otimes c^{i}$ denotes a sum of elementary tensors over the index $i$.
\end{itemize}
\end{notn}

\subsection{The dg cobar construction and its extensions}

Let $\Om$ denote the \emph{cobar construction} functor from $\cat {Coalg}_{\Bbbk}$ to $\cat {Alg}_{\Bbbk}$, defined by 
$$\Om C= \left(T (\si \overline C), d_{\Om}\right)$$
where, if $d$ denotes the differential on $C$, then
\begin{align*}
d_{\Om}(\si c_{1}|\cdots|\si c_{n})=&\sum _{1\leq j\leq n}\pm \si c_{1}|\cdots |\si (dc_{j})|\cdots |\si c_{n}\\ 
&+\sum _{1\leq j\leq n}\pm \si c_{1}|...|\si c_{ji}|\si c_{j}{}^{i}|\cdots |\si c_{n},
\end{align*}
with signs determined by the Koszul rule, where the reduced comultiplication applied to $c_{j}$ is $c_{ji}\otimes c_{j}{}^{i}$.  A straightforward computation shows that $\Om C$ is isomorphic to the totalization of the cosimplicial cobar construction if $C$ is $1$-connected (i.e., $C$ is connected and $C_{1}=0$).

The graded vector space underlying $\Om C$ is naturally a free associative algebra, with multiplication given by concatenation. The differential $d_{\Om }$ is a derivation with respect to this concatenation product, so that $\Om C$ is itself a chain algebra.  Any chain algebra map $\alpha:\Om C\to A$ is determined by its restriction to the algebra generators $\si \overline C$. 

The following two extensions of the cobar construction play an important role below.   Let $\cat {Mix}_{C,\Om C}$ and $\cat {Mix}_{\Om C,C}$ denote the categories of left $C$-comodules in the category of right $\Om C$-modules and of right $C$-comodules in the category of left $\Om C$-modules, respectively.  We call the objects of these categories \emph{mixed modules}.  There are functors
$$\P_{L}: \cat {Coalg}_{\Bbbk}\to \cat {Mix}_{C, \Om C}\quad\text{and}\quad \P_{R}: \cat {Coalg}_{\Bbbk}\to \cat {Mix}_{ \Om C,C},$$
which we call the \emph{left and right based path constructions} on $C$ (where left and right refer to the side of the $C$-coaction) and which are defined as follows.
$$\P_{L}C= \left( C\otimes T (\si \overline C), d_{\P_{L}}\right)\quad\text{and}\quad \P_{R}C= \left(T (\si \overline C)\otimes C, d_{\P_{R}}\right),$$
where
\begin{align*}
d_{\P_{L}}(e\otimes \si c_{1}|\cdots|\si c_{n})=&de\otimes \si c_{1}|\cdots|\si c_{n}\;\pm e\otimes d_{\Om}(\si c_{1}|\cdots|\si c_{n})\\
&\pm e_{j}\otimes \si e^{j}|\si c_{1}|\cdots|\si c_{n}\\
\\
d_{\P_{R}}(\si c_{1}|\cdots|\si c_{n}\otimes e)=&d_{\Om}(\si c_{1}|\cdots|\si c_{n})\otimes e\;\pm \si c_{1}|\cdots|\si c_{n}\otimes de\\
&\pm \si c_{1}|\cdots|\si c_{n}|\si e_{j}\otimes e^{j},
\end{align*}
where $\Delta (e)=e_{j}\otimes e^{j}$, and applying $\si$ to an element of degree 0 gives $0$.
For every $C$ in $\cat{Coalg}_{\Bbbk}$, there are twisted extensions of chain complexes
$$\xymatrix@1{\Om C\cof^{\eta \otimes 1} &\P_{L}C\fib^{1\otimes \ve} &C&&&\Om C\cof^{1\otimes \eta} &\P_{R}C\fib^{\ve\otimes 1} &C,}
$$
which are dg analogues of the based pathspace fibration, where $\eta:\Bbbk \to C$ is the coaugmentation and $\ve: \Om C \to \Bbbk$ the obvious augmentation.

As proved in \cite[Proposition 10.6.3]{neisendorfer}, both $\P_{L}C$ and $\P_{R}C$ are homotopy equivalent to the trivial mixed module $\Bbbk$, for all $C$ in $\cat {Coalg}_{\Bbbk}$, via a chain homotopy defined in the case of $\P_{R}C$ by

{\small $$h_{R}: \P_{R}C\to \P_{R}C: w\otimes e \mapsto \begin{cases} 0 &: |e|>0 \text{ or } w=1 \\ \si c_{1}|\cdots | \si c_{n-1}\otimes c_{n} &: |e|=0 \text{ and } w= \si c_{1}|\cdots | \si c_{n}\end{cases}$$}

\noindent and analogously in the case of $\P_{L}C$.  Observe that, when restricted to the sub $\Om C$-module of elements in positive degree, $h_{R}$ is a homotopy of left $\Om C$-modules, while $h_{L}$ is a homotopy of right $\Om C$-modules.

The proposition below generalizes this contractibility result.

\begin{prop}\label{prop:SDR} There are strong deformation retracts
\begin{enumerate}
\item $\sdr{\Om C}{\P_{R}C\square_{C} \P_{L}C}{\sigma}{\pi}$
in the category of left $\Om C$-modules, and 
\item $\sdr{C}{\P_{L}C \otimes_{\Om C}\P_{R}C}{\iota}{\rho}$
in the category of left $C$-comodules.
\end{enumerate}
\end{prop}

\begin{proof} Note that the graded vector space underlying $\P_{R}C\square_{C} \P_{L}C$ is isomorphic to $T(\si \overline C)\otimes C \otimes T(\si \overline C)$, while that underlying $\P_{L}C \otimes_{\Om C}\P_{R}C$ is isomorphic to $C\otimes T(\si \overline C) \otimes C$.

In the $\Om C$-module case, we define left $\Om C$-module maps
$$\pi:\P_{R}C\square_{C} \P_{L}C \to \Om C: v\otimes c \otimes w\mapsto \begin{cases} \ve(c)\cdot vw&: |c|=0\\ 0&: |c|\not=0,\end{cases}$$
where $\ve: C\to \Bbbk$ denotes the counit, and  
$$\sigma: \Om C \to \P_{R}C\square_{C} \P_{L}C: w \mapsto w\otimes 1 \otimes 1.$$
While it is obvious that $\pi\sigma$ is the identity, showing that $\sigma\pi$ is homotopic to the identity requires  a new chain homotopy $h: \P_{R}C\square_{C} \P_{L}C \to \P_{R}C\square_{C} \P_{L}C$  defined by
{\small $$h(1\otimes c\otimes \si c_{1}| \cdots | \si c_{n})=\begin{cases} \sum _{1\leq i\leq n}\pm \si c_{1}|...|\si c_{i-1}\otimes c_{i}\otimes \si c_{i+1}| \cdots | \si c_{n} &: |c|=0\\ 0&: |c|\not=0,\end{cases}$$}

\noindent then extended to a map of left $T(\si \overline C)$-modules.
A straightforward computation shows that $Dh +hD= \id - \sigma\pi$ as desired, where $D$ denotes the differential on $\P_{R}C\square_{C} \P_{L}C$.  

Let $\Delta$ denote the comultiplication on $C$. In the $C$-comodule case, we define left $C$-comodule maps by
$$\iota: C\to \P_{L}C \otimes_{\Om C}\P_{R}C: c \mapsto c_{i}\otimes 1 \otimes c^{i},$$
where $\Delta (c)=c_{i}\otimes c^{i}$,
and  
$$\rho:  \P_{L}C \otimes_{\Om C}\P_{R}C \to C: c \otimes w \otimes c' \mapsto \begin{cases} c &: |w|=|c'|=0\\ 0&: \text{else.}\end{cases}$$
It is obvious that $\rho\iota$ is equal to the identity and that $\iota \rho$ is chain homotopic to the identity as left $C$-comodules, via the chain homotopy $\id_{\P_{L}C}\otimes_{\Om C}h_{R}$, which itself respects the left $C$-coaction as well.
\end{proof}

Our interest in the left and right based path constructions stems from the following proposition.

\begin{prop}\label{prop:bp-adjunction} The pair of functors
$$\hugeadjunction{\cat{Comod}_{C}}{\cat{Mod}_{\Om C}}{-\square_{C}\P_{L}C}{-\otimes _{\Om C} \P_{R}C}$$
forms a dg-adjunction.
\end{prop}

\begin{proof} It is well known that the dg-enrichments of $\cat{Comod}_{C}$ and $\cat{Mod}_{\Om C}$ can be constructed as equalizers in $\cat {Ch}_{\Bbbk}$, as follows.  For right $C$-comodules $N$ and $N'$ with $C$-coactions $\rho$ and $\rho'$,
$$\underline{\cat{Comod}}\,_{C}(N,N')=\lim \left(\underline{\cat {Ch}}\,_{\Bbbk} (N,N')   \egal{\rho'_{*}}{\rho^{*}\circ (-\otimes C)} \underline{\cat {Ch}}\,_{\Bbbk} (N,N'\otimes C) \right),$$
where the underline denotes the hom-chain complex, as opposed to the hom-set.
For right $\Om C$-modules $M$ and $M'$ with $\Om C$-actions $\alpha$ and $\alpha'$,
$$\underline{\cat{Mod}}\,_{\Om C}(M,M')=\lim \left(\underline{\cat {Ch}}\,_{\Bbbk} (M,M')   \egal{\alpha'_{*}\circ (-\otimes \Om C)}{\alpha^{*}} \underline{\cat {Ch}}\,_{\Bbbk} (M\otimes \Om C, M') \right).$$
It is easy to see from these constructions that both of the functors in the statement of the proposition are dg-enriched, essentially because on the underlying graded vector spaces, both functors are given by tensoring with a fixed object.

Since $\pi:\P_{R}C\square_{C} \P_{L}C \to \Om C$ is actually a map of $\Om C$-bimodules, while $\iota: C\to \P_{L}C \otimes_{\Om C}\P_{R}C$ is a map of $C$-bicomodules, there are dg-natural transformations
$$\id \cong -\square _{C}C \xrightarrow {-\square _{C}\iota} -\square_{C}(\P_{L}C \otimes_{\Om C}\P_{R}C)$$
and 
$$-\otimes _{\Om C}(\P_{R}C\square_{C} \P_{L}C) \xrightarrow {-\otimes _{\Om C} \pi} -\otimes_{\Om C}\Om C \cong \id,$$
which provide the unit and counit of the adjunction.  It is an easy exercise to verify the triangle inequalities.  
\end{proof}

The extension of the cobar construction that is the focus of this article is the  \emph{coHochschild complex} functor
$$\cohoch: \cat{Coalg}_{\Bbbk}\to \cat {Ch}_{\Bbbk},$$ 
defined as follows \cite{hps}.  
Let $C$ be a connected, coaugmented chain coalgebra with comultiplication $\Delta(c)=c_{i}\otimes c^{i}$. We then let 
$$\cohoch(C)= \left(C\otimes T (\si \overline C), d_{\cohoch}\right)$$
where
\begin{align*}
d_{\cohoch}(e\otimes \si c_{1}|\cdots|\si c_{n})=&de\otimes \si c_{1}|\cdots|\si c_{n}\;\pm e\otimes d_{\Om}(\si c_{1}|\cdots|\si c_{n})\\
&\pm e_{j}\otimes \si e^{j}|\si c_{1}|\cdots|\si c_{n}\\
& \pm e^{i}\otimes \si c_{1}|\cdots|\si c_{n}|\si e_{i},
\end{align*}
where $\Delta (e)=e_{j}\otimes e^{j}$, and applying $\si$ to an element of degree 0 gives $0$. The signs follow from the Koszul rule, as usual.  As in the case of the cobar construction, it is not hard to show that $\cohoch(C)$ is isomorphic to the totalization of a certain cosimplicial construction when $C$ is $1$-connected; {see the analogue for spectra in Section \ref{sec-spec}.}

For every $C$ in $\cat{Coalg}_{\Bbbk}$, there is  a twisted extension of chain complexes
\begin{equation}\label{eqn:cohoch-ext}
\xymatrix@1{\Om C\cof^{\eta \otimes 1} &\cohoch (C)\fib^{1\otimes \ve} &C,}
\end{equation}
which is the dg analogue of the free loop fibration, where, as above, $\eta:\Bbbk \to C$ is the coaugmentation and $\ve: \Om C \to \Bbbk$ the obvious augmentation.

\begin{rmk}  There is a natural and straightforward extension of the coHochschild complex of a chain coalgebra to a \emph{cocyclic complex}, analogous to the extension of the Hochschild complex of a chain algebra to the cyclic complex.  Moreover the construction of the coHochschild complex of a coalgebra $C$ can be generalized to allow for coefficients in any $C$-bicomodule \cite[Section 1.3]{hps}.
\end{rmk} 

\subsection{Properties of the dg coHochschild construction}

The result below provides a first indication of the close link between the Hochschild and coHochschild constructions.

\begin{prop}\label{prop:cohoch-hoch}\cite[Corollary 2.22]{hess.twisted}
Let $\hoch: \cat {Alg}_{\Bbbk}\to \cat {Ch}_{\Bbbk}$ denote the usual Hochschild construction. For any $C$ in $\cat {Coalg}_{\Bbbk}$, there is a natural quasi-isomorphism $$\cohoch(C) \xrightarrow{\simeq} \hoch(\Om C).$$
\end{prop}

\begin{rmk} Another (and easier) way to obtain an algebra from a coalgebra $C$ is to take its linear dual degreewise, denoted $C^{\vee}$. A straightforward computation shows that for any $C$ in $\cat {Coalg}_{\Bbbk}$, 
the linear dual of the coHochschild complex of $C$ is isomorphic to the Hochschild (cochain) complex of $C^{\vee}$, i.e.,
$$\big(\cohoch(C)\big)^{\vee}\cong \hoch (C^{\vee}).$$
\end{rmk}

We show below that Proposition \ref{prop:cohoch-hoch} can be categorified, i.e., lifted to model categories of $C$-comodules and $\Om C$-modules. We then use this categorification to establish ``agreement'' and Morita invariance for coHochschild homology.

\begin{convention}  Henceforth,  we fix the following model structures.   
\begin{itemize}
\item Endow $\cat{Ch}_{\Bbbk}$  with the model structure for which cofibrations are degreewise injections, fibrations are degreewise surjections, and weak equivalences are quasi-isomorphisms. 
\item For any $C$ in $\cat {Coalg}_{\Bbbk}$, the category $\cat {Mod}_{\Om C}$ of right $\Om C$-modules is equipped with the model structure right-induced from  $ \cat {Ch}_{\Bbbk}$ by the forgetful functor $$U: \cat {Mod}_{\Om C}\to \cat {Ch}_{\Bbbk},$$ which exists by \cite[4.1]{moneq}. The fibrations in $\cat {Mod}_{\Om C}$ are exactly those module maps that are degreewise surjective, whence every object is fibrant.  Every cofibration in $\cat {Mod}_{\Om C}$ is a retract of a sequential colimit of module maps given by pushouts along morphisms of the form (injection)$\otimes \Om C$.
\item For any $C$ in $\cat {Coalg}_{\Bbbk}$,  the category $\cat {Comod}_{C}$ of right $C$-comodules is equipped with the model structure left-induced  from $\cat {Ch}_{\Bbbk}$ by the forgetful functor $$U: \cat {Comod}_{C}\to \cat {Ch}_{\Bbbk},$$ which exists by \cite[6.3.7]{hkrs}; {see also~\cite{gkr}.}   The cofibrations in $ \cat {Comod}_{C}$ are exactly those comodule maps that are degreewise injective, whence every object is cofibrant.  Every retract of a sequential limit of comodule maps given by pullbacks along morphisms of the form (surjection)$\otimes C$ is a fibration in $\cat{Comod}_{C}$.
\end{itemize}
\end{convention}

\subsubsection{Categorifying Proposition \ref{prop:cohoch-hoch}}

\begin{prop}\label{prop:cobar}  For any $C$ in $\cat{Coalg}_{\Bbbk}$, the enriched adjunction
$$\hugeadjunction{\cat{Comod}_{C}}{\cat{Mod}_{\Om C}}{-\square_{C}\P_{L}C}{-\otimes _{\Om C} \P_{R}C}$$
is a Quillen equivalence.
\end{prop}

\begin{proof}  To simplify notation, we write 
$$L_{C}=-\square_{C}\P_{L}C\quad\text{and}\quad R_{C}=-\otimes _{\Om C}\P_{R}C.$$
Observe that $R_{C}(M)\cong (M\otimes C, D_{R})$  and $L_{C}(N)\cong (N\otimes \Om C, D_{L})$, for every  $M$ in $\cat {Mod}_{\Om C}$ and every $N$ in $\cat{Comod}_{C}$, where the differentials of these complexes are specified by
$$D_{R}(x\otimes c)= dx \otimes c \;\pm  x \otimes dc \;\pm (x \cdot \si c_{i})\otimes c^{i}$$
and
$$D_{L}(y\otimes w)=  dy \otimes w \;\pm y \otimes d_{\Om}w \;\pm y_{j}\otimes (\si c^{{j}}\cdot w),$$
Here, $\rho (y)=y\otimes + y_{j}\otimes c^{j}$ and $\Delta (c)=c_{i}\otimes c^{i}$,  where $\rho$ is the $C$-coaction on $N$ and $\Delta$ the comultiplication on $C$, and the signs are determined by the Koszul rule.

We show first that $L_{C}\dashv R_{C}$ is a Quillen adjunction.   If $j:N\to N'$ is a cofibration in $\cat {Comod}_{C}$, i.e., a degreewise injective morphism of $C$-comodules, then  there is a decomposition $N'=N\oplus V$ as graded vector spaces, since we are working over a field. It follows $L_{C}(N')$ can be built inductively as an $\Om C$-module  from $L_{C}(N)$, as we explain below.
  
Suppose that $dv\in N$ and $\rho (v)-v\otimes 1\in N\otimes C$ for all $v\in V$, and let $B$ be a basis of $V$. There is a pushout diagram in $\cat {Mod}_{\Om C}$
$$\xymatrix{\coprod _{x\in B}S^{|x|-1}\otimes \Om C\ar [d]\ar[r]^(0.6){\iota}&L_{C}(N)\ar[d]^{L_{C}(j)}\\
\coprod _{x\in B}D^{|x|}\otimes \Om C\ar[r]&L_{C}(N\oplus V)},$$
where $ S^{m}$ is the chain complex with only one basis element, which is in degree $m$, $D^{m+1}$ has two basis elements, in degrees $m$ and $m+1$, with a differential linking the latter to the former, and $\iota$ maps the generator of $S^{|x|-1}$ to $dx$ for every $x\in B$. It follows that $L_{C}(j)$ is a cofibration, in this special case.  

In the general case, we use that any comodule is the filtered colimit of its finite-dimensional subcomodules \cite[Lemma 1.1]{getzler-goerss}. We can structure this filtered colimit more precisely as follows.  For any $n\geq 1$,  let $\{N'(i,n)\mid i\in \mathcal I_{n}\}$ denote the set of subcomodules of $N'$ such that $ N'(i,n)/N$ is of dimension $n$ for all $i\in \mathcal I_{n}$, and set $N'(n)=\Sigma_{i\in \mathcal I_{n}}N'(i,n)$.  The argument above shows that the injection $N\to N'(1)$ is a pushout along a morphism of the form (injection)$\otimes \Om C$, and, more generally, that the inclusion $N'(n) \to N'(n+1)$ is a pushout along a morphism of the form (injection)$\otimes \Om C$ for all $n$, whence $j\colon N\to N'=\colim_{n}N'(n)$ is a cofibration.

On the other hand, we can show by a spectral sequence argument that the functor $L_{C}$ preserves all weak equivalences and therefore preserves trivial cofibrations.  Any $N$ in $\cat{Comod}_{C}$ 
admits a natural ``primitive'' filtration 
\begin{equation}\label{eqn:filtration}
F_{0}N\subseteq F_{1}N\subseteq F_{2}N\subseteq\cdots \subseteq N
\end{equation}
as a $C$-comodule, i.e., $F_{0}N=\ker ( N\xrightarrow{\bar\rho} N\otimes C)$ and 
$$F_{m}N=\ker (N\xrightarrow {\bar\rho^{(n)}} N\otimes C^{\otimes n} )$$
for all $m\geq 1$, where $\bar \rho=\rho- N\otimes \eta$, and $\bar \rho^{(n)}=(\bar \rho\otimes C^{\otimes n-1})\bar \rho^{(n-1)}$.  Note that this filtration is always exhaustive, since $C$ is connected, and $(\rho \otimes C)\rho= (N\otimes \Delta)\rho$.

Consider the exhaustive filtration of $L_{C}(N)$ as an $\Om C$-module induced by applying $L_{C}$ to the primitive filtration (\ref{eqn:filtration}) of $N$:
{\small $$(F_{0}N\otimes \Om C, d\otimes 1 + 1\otimes d_{\Om}) \subseteq (F_{1}N\otimes \Om C, D_{L})\subseteq (F_{2}(N)\otimes \Om C, D_{L})\subseteq\cdots \subseteq (N\otimes \Om C, D_{L}).$$}
The  $E_{2}$-term of the spectral sequence associated to this filtration, which converges to $H_{*}L_{C}(N)$, is isomorphic as a graded vector space to $H_{*}(N)\otimes H_{*}(\Om C)$, from which it follows that a quasi-isomorphism of $C$-comodules induces an isomorphism on the $E_{2}$-terms of the associated spectral sequences and thus on the $E_{\infty}$-terms as well.

Since $R_{C}(M)$ is cofibrant in $\cat{Comod}_{C}$ and $L_{C}(N)$ is fibrant in $\cat{Mod}_{\Om C}$ for every  $M$ in $\cat {Mod}_{\Om C}$ and every $N$ in $\cat{Comod}_{C}$, it follows from \cite[Proposition 1.3.13(b)]{hovey} that $L_{C}\dashv R_{C}$ is a Quillen equivalence if the unit $N\to R_{C}L_{C}N$ and counit $L_{C}R_{C}M \to M$ of the $L_{C}\dashv R_{C}$ adjunction are weak equivalences for every $N$ (since they are all cofibrant) and every $M$ (since they are all fibrant). To conclude, it suffices therefore to observe that for every $M$, there is a sequence of isomorphisms and weak equivalences in $\cat {Mod}_{\Om C}$,
$$L_{C}R_{C}M= M\otimes _{\Om C} \P_{R}C\square_{C}\P_{L}C \simeq M\otimes_{\Om C} \Om C\cong M,$$
where the weak equivalence is a consequence of Proposition \ref{prop:SDR}(1), and that for every $N$, there is a sequence of isomorphisms and weak equivalence in $\cat{Comod}_{C}$,
$$R_{C}L_{C}N=N\square_{C}\P_{L}C\otimes _{\Om C} \P_{R}C\simeq N\square_{C}C\cong N,$$
where the weak equivalence follows from Proposition \ref{prop:SDR}(2).
\end{proof}

\begin{rmk} In Chapter 2 of his thesis \cite{lefevre-hasegawa}, Lef\`evre-Hasegawa defined a model structure on the category of \emph{cocomplete} comodules over a \emph{cocomplete}, coaugmented dg coalgebra $C$, of which the proposition above would seem to be a special case.  Here, ``cocomplete'' means that the respective primitive filtration is exhaustive, which is not immediate if $C$ is not connected.  It seems, however, that Lef\`evre-Hasegawa did not check that the category of cocomplete comodules is closed under limits, which we suspect is actually not true.

The proposition above also could be viewed as a special case \cite [Proposition 3.15]{drummond-cole-hirsh}, which establishes a general Quillen equivalence between certain categories of coalgebras over a cooperad and algebras over an operad as mediated by a twisting morphism.  One needs to show that the weak equivalences of \cite {drummond-cole-hirsh} are the same as those in our model structure on $\cat{Comod}_{C}$, which follows from Proposition \ref{prop:SDR}.  We think there is merit in providing an explicit, independent proof in this special case, especially as it makes evident  the ``geometric'' nature of the proof (using based path spaces).
\end{rmk}

\begin{ex}   Proposition \ref{prop:cobar} implies that for every reduced simplicial set $K$, 
$$\giantadjunction{\cat{Comod}_{C_{*}(K)}}{\cat{Mod}_{\Om C_{*}(K)}}{-\square_{C_{*}(K)}\P_{L}C_{*}(K)}{-\otimes _{\Om C_{*}(K)}\P_{R}C_{*}(K)}$$
is a Quillen equivalence, where $C_{*}(K)$ denotes the normalized chain coalgebra of $K$ with coefficients in $\Bbbk$.  Moreover, if $K$ is actually $1$-reduced, there is a natural quasi-isomorphism of chain algebras $$\alpha_{K}\colon \Om C_{*}(K) \xrightarrow \simeq C_{*}(\G K)$$ \cite{szczarba}, where $\G$ denotes the Kan loop group functor, which induces a Quillen equivalence 
$$\bigadjunction{\cat{Mod}_{\Om C_{*}(K)}}{\cat{Mod}_{C_{*}(\G K)}}{\alpha_{!}} {\alpha^{*}}.$$
It follows that if $K$ is $1$-reduced, there is a Quillen equivalence 
$$\bigadjunction{\cat{Comod}_{C_{*}(K)}}{\cat{Mod}_{C_{*}(\G K)}}{} {}.$$
\end{ex}

\subsubsection{Agreement}

The Quillen equivalence of {Proposition~\ref{prop:cobar}} enables us to establish ``agreement" for coHochschild homology, analogous to ``agreement" for Hoch\-schild homology, which we recall now. In \cite{mccarthy}, McCarthy extended the notion of Hochschild homology in a natural way to exact categories, seen as ``rings with many objects,''  and established  ``agreement'' in this context:  the Hochschild homology of the exact category of finitely generated projective modules over a ring $R$ is isomorphic to the Hochschild homology of $R$ itself.    

Keller generalized the definition of Hochschild homology to dg categories (seen as dg algebras with many objects) in \cite{keller:cyclic}, \cite{keller:icm} and showed that agreement still held in this more general setting: for any dg $\Bbbk$-algebra $A$, the Hochschild homology of $A$ is isomorphic to that of the full dg subcategory $\cat{dgfree}_{A}$ of $\cat {Mod}_{A}$  \cite[Theorem 2.4]{keller:cyclic}.  The objects of $\cat{dgfree}_{A}$  are finitely generated quasi-free $A$-modules, i.e., $A$ -modules such that the underlying nondifferential graded module is free and finitely generated over the nondifferential graded algebra underlying $A$.  Note that objects in $\cat{dgfree}_{A}$ are cofibrant in our chosen model structure on $\cat {Mod}_{A}$.

For any $C$ in $\cat {Coalg}_{\Bbbk}$, the notion of agreement for coHochschild homology is expressed in terms of the full dg subcategory $\cat{dgcofree}_{C}$ of $\cat {Comod}_{C}$, the objects of which are the fibrant $C$-comodules $N$ such that there exists a quasi-isomorphism of $\Om C$-modules $L_{C}(N) \xrightarrow \simeq M$, where $M$ is an object of $\cat {dgfree}_{\Om C}$.  Observe that  a morphism $L_{C}(N) \to M$ is a quasi-isomorphism if and only if its transpose $N\to R_{C}(M)$ is a quasi-isomorphism, since $L_{C}\dashv R_{C}$ is a Quillen equivalence, all objects in $\cat{Comod}_{C}$ or in the image of $R_{C}$ are cofibrant, and all objects in $\cat{Mod}_{\Om C}$ or in the image of $L_{C}$ are fibrant.

\begin{rmk}  The notation for the dg subcategory $\cat{dgcofree}_{C}$ is a bit abusive, since not all of its objects are actually quasi-cofree, i.e., such that the underlying nondifferential graded comodule is cofree over the nondifferential graded coalgebra underlying $C$.  On the other hand, any quasi-cofree $C$-comodule that is ``finitely cogenerated'' over $C$ is an object of $\cat{dgcofree}_{C}$.  More precisely, any quasi-cofree $C$-comodule is the limit of a tower  of comodule maps given by pullbacks along morphisms of the form (surjection)$\otimes C$ and is therefore fibrant. Moreover, if $N$ is ``finitely cogenerated'' by $V$, then its image under $L_{C}$ is finitely generated by $V$.
\end{rmk}

\begin{prop} \label{prop:cohoch-props}  Agreement holds for coHochschild homology of coalgebras, i.e., for every $C$ in $\cat {Coalg}_{\Bbbk}$, $$H_{*}\big(\cohoch (C)\big) \cong H_{*}\big(\hoch(\cat{dgcofree}_{C})\big).$$
\end{prop}

The key to the proof of agreement for coalgebras, as well as to establishing Morita invariance at the end of this section, is the following lemma, providing conditions under which dg Quillen equivalences induce quasi-equivalences of dg categories.  We were unable to find this result in the literature, though we suspect it is well known.

\begin{lem}\label{lem:q-e} Let $\adjunction{\cat M}{\cat N}{F}{G}$ be an enriched Quillen equivalence of dg model categories.  Let $\cat M'$ and $\cat N'$ be full dg subcategories of $\cat M$ and $\cat N$, respectively, where all objects of $\cat M'$ and of $\cat N'$ are fibrant and cofibrant.   
\begin{enumerate}
\item Suppose that $F$ restricts and corestricts to a functor $F':\cat M' \to \cat N'$. If for every $Y\in \ob \cat N'$ there exists a weak equivalence $p_Y:\widehat{G(Y)} \xrightarrow \simeq {G(Y)}$ where $\widehat {G(Y)}\in \ob \cat M'$, then $F'$ is a quasi-equivalence.
\item Suppose that $G$ restricts and corestricts to a functor $G':\cat N' \to \cat M'$. If for every $X\in \ob \cat M'$ there exists a weak equivalence $j_X:{F(X)} \xrightarrow \simeq \widehat{F(X)}$ where $\widehat{F(X)}\in \ob \cat N'$,  then $G'$ is a quasi-equivalence.
\end{enumerate}
\end{lem}

\begin{proof} We prove (1) and leave the dual proof of (2) to the reader.  By \cite[Proposition 1.3.13]{hovey}, since $F\dashv G$ is a Quillen equivalence, the unit morphism $\eta_{X'}:X' \to G'F'(X')$ is a weak equivalence for every $X'\in \ob \cat M'$, since $X'$ is cofibrant and $F'(X')$ is fibrant.   It follows that $$F'_{X',Y'}: \hom _{\cat M'}(X',Y') \to \hom _{\cat N'}\big( F'(X'), F'(Y')\big)$$ is a quasi-isomorphism for all $X',Y'\in \ob \cat M'$, since it factors as
$$\hom_{\cat M'}(X',Y') \xrightarrow {\simeq} \hom_{\cat M'}\big(X', G'F'(Y')\big)\cong \hom _{\cat N'}\big(F'(X'), F'(Y')\big).$$
The first map above is a quasi-isomorphism because $X'$ is cofibrant, and $\eta_{X'}$ is a weak equivalence between fibrant objects.  We conclude that $F'$ is quasi-fully faithful.

Let $Z'\in \ob \cat N'$.   Since we can choose $p_{Z'}$ as  a cofibrant replacement of $G'(Z')$, the composite 
$$F'\big( \widehat{G(Z')}) \xrightarrow {F'(p_{Z'})} F'G'(Z') \xrightarrow {\ve_{Z'}} Z'$$
is a model for the derived counit of the adjunction and therefore a weak equivalence.  Both its source and target are objects in $\cat N'$ and therefore fibrant (and cofibrant), whence $\hom _{\cat N'} \big(W', \ve_{Z'}F'(p_{Z'})\big)$ is a quasi-isomorphism for all $W'\in \ob\cat N'$, as all objects in $\cat N'$ are cofibrant.  By Exercise 6 in \cite[Section 2.3]{toen}, it follows that the homology class of $\ve_{Z'}F'(p_{Z'})$ is an isomorphism in the homology category of $\cat N'$ and thus that $F'$ is quasi-essentially surjective.
\end{proof}

\begin{proof}[Proof of Proposition \ref{prop:cohoch-props}]  Observe that
$$\cohoch (C) \simeq \hoch (\Om C) \simeq \hoch ( \cat {dgfree}_{\Om C})\simeq \hoch ( \cat {dgcofree}_{C}),$$
where the first equivalence is given by Proposition \ref{prop:cohoch-hoch} and the second by \cite{keller:cyclic}. The third follows from Proposition  \ref{prop:cobar} and Lemma \ref{lem:q-e}(2).  To see that all of the conditions of Lemma \ref{lem:q-e}(2) are satisfied, note first that all objects of $\cat {dgfree}_{\Om C}$ and $\cat {dgcofree}_{C}$ are both fibrant and cofibrant.  Moreover, the functor $R_{C}$ restricts and corestricts to a dg functor from $\cat{dgfree}_{\Om C}$ to  $\cat{dgcofree}_{C}$, since the counit $L_{C}R_{C}\to \id$ of the Quillen equivalence $L_{C}\dashv R_{C}$ is a quasi-isomorphism on objects that are cofibrant and fibrant.  Finally, $\cat{dgcofree}_{C}$ is defined precisely so that the remaining condition holds as well.

 Because every quasi-equivalence of dg categories is a Morita equivalence \cite[Section 4.4]{toen}, and Hochschild homology of dg categories is an invariant of Morita equivalence \cite[Section 5.2]{toen}, we can conclude.
\end{proof}

\subsubsection{Morita invariance}
Thanks to the ``agreement'' result established above, we can now show that coHochschild homology satisfies a property dual to the Morita invariance of Hochschild homology. The study of equivalences between categories of comodules over coalgebras over a field, commonly referred to as \emph{Morita-Takeuchi theory}, was initiated by Takeuchi \cite{takeuchi} and further elaborated and generalized by Farinati and Solotar \cite{farinati-solotar} and Brzezinski and Wisbauer \cite{brzezinski-wisbauer}, among others.  In \cite{berglund-hess}, Berglund and Hess formulated a homotopical version of this theory, in terms of the following notion.

\begin{defn} \label{def:braided bimodule}  Let $C,D$ be objects in $\cat {Coalg}_{\Bbbk}$.
A \emph{braiding from $C$ to $D$} is a pair $(X,T)$ where $X$ is in $\cat {Ch}_{\Bbbk}$, and $T$ is a morphism of chain complexes
$$T\colon C\otimes X\rightarrow X\tensor D$$
satisfying the following axioms.

\noindent {\bf (Pentagon axiom)}

The diagram
\begin{equation} \label{eq:pentagon}
\xymatrix{C\tensor X \ar[d]_-{\Delta_C\tensor 1} \ar[rr]^-{T} && X\tensor D \ar[d]^-{1\tensor \Delta_D} \\
C\tensor C\tensor X \ar[dr]_-{1\tensor T} && X\tensor D \tensor D \\
& C\tensor X\tensor D \ar[ur]_-{T\tensor 1}}
\end{equation}
commutes.

\vskip5pt
\noindent {\bf (Counit axiom)}

The diagram
\begin{equation} \label{eq:counit}
\xymatrix{C\tensor X \ar[rr]^-{T} \ar[d]^-{\epsilon_C\tensor 1} && X\tensor D \ar[d]^-{1\tensor \epsilon_D} \\
\Bbbk \tensor X \ar[r]^-\cong & X & X\tensor \Bbbk \ar[l]_-\cong}
\end{equation}
commutes.

We write $(X,T)\colon C \to D$ to indicate that $(X,T)$ is a  braiding from $C$ to $D$.
\end{defn}

\begin{ex}[Change of coalgebras] \label{ex:braided bimodules 3}
A morphism  $f\colon C\to D$ in $\cat {Coalg}_{\Bbbk}$ gives rise to a braiding $(\Bbbk, f):C\to D$ and thus to an adjunction
\begin{equation}\label{eq:corestriction}
\bigadjunction{\cat{Comod}_{C}}{\cat {Comod}_{D}}{f_{*}}{f^{*}},
\end{equation}
which we call the \emph{coextension/corestriction-of-scalars adjunction} or \emph{change-of-corings adjunction} associated to $f$. The $D$-component of the counit of the $f_{*}\dashv f^{*}$ adjunction is $f$ itself and that for every $C$-comodule $(M, \delta)$,
$$f_{*}(M,\delta)= \big(M, (1\otimes f)\delta\big).$$
\end{ex}

Since $\cat{Comod}_{C}$ is bicomplete for all coalgebras $C$ (as $\cat {Ch}_{\Bbbk}$ is locally presentable, and $\cat{Comod}_{C}$ is a category of coalgebras for the comonad $-\otimes C$), it follows from   \cite[Proposition 3.17]{berglund-hess} that every braiding $(X,T)\colon C\to D$ gives rise to a $\cat {Ch}_{\Bbbk}$-adjunction
$$\bigadjunction{\cat {Comod}_{C}}{\cat {Comod}_{D}}{T_{*}}{T^{*}},\quad T_{*} \dashv T^{*},$$
such that the diagram 
$$\xymatrix{\cat {Comod}_{C} \ar[rr]^{T_{*}} \ar[d]_{U} && \cat {Comod}_{D} \ar[d]^{U}\\
\cat {Ch}_{\Bbbk} \ar[rr]^{-\tensor X} && \cat{Ch}_{\Bbbk}}
$$
commutes, i.e., the endofunctor on $\cat {Ch}_{\Bbbk}$ is just given by tensoring with $X$.  Moreover, since we are working over  a field and thus tensoring with any chain complex preserves both weak equivalences and equalizers, Proposition 3.31 in \cite{berglund-hess} implies that if $X$ is dualizable with dual $X^{\vee}$, then $T^{*}=-\square_{D}(X^{\vee}\tensor C)$, where $-\square_{D}-$ denotes the cotensor product over $D$.

The following lemma, relating braidings to the adjunction $L_{C}\dashv R_{C}$ studied above, plays an important role at the end of this section.

\begin{lem}\label{lem:braiding-compatible} For every braiding $(X,T)\colon C\to D$, the dg adjunction
$$\giantadjunction {\cat{Mod}_{\Om C}}{\cat {Mod}_{\Om D}}{-\otimes_{\Om C}(X \otimes \Om D)} {\hom_{\Om D}(X\otimes \Om D, -)}$$
is a dg Quillen pair, satisfying the natural isomorphisms 
$$\big (-\otimes_{\Om C}(X \otimes \Om D)\big)\circ L_{C}\cong L_{D}\circ T_{*}$$
and
$$R_{C}\circ \hom_{\Om D}(X\otimes \Om D, -)\cong T^{*}\circ R_{D}.$$
\end{lem}

\begin{proof} Given our choice of model structures  on module categories, it is easy to check that the adjunction above is indeed a Quillen pair. It suffices to establish the first isomorphism, since the second  is then an immediate consequence.  The computation is straightforward, given that  
the  left $\Om C$-action on $X\otimes \Om D$ induced by the braiding
$$T\colon C\otimes X \to X\otimes D: c\otimes x \mapsto x_{i}\otimes d^{i}$$
is specified by $\si c\cdot (x \otimes w)= x_{i}\otimes \si d^{i}\cdot w$ for all $c\in \overline C$, $x\in X$, and $w\in \Om D$.
\end{proof}

The coalgebraic analogue of Morita equivalence is defined as follows. 

\begin{defn} Let $C, D$ be in $\cat{Coalg}_{\Bbbk}$.    If there is a braiding $(X,T)$ from $C$ to $D$ such that $T_{*} \dashv T^{*}$ is a Quillen equivalence, then $C$ and $D$ are \emph{(homotopically) Morita-Takeuchi equivalent}.
\end{defn}

As a special case of \cite [Theorem 4.16]{berglund-hess}, we can describe Morita-Takeuchi-equivalent pairs of chain coalgebras in terms of the following notions, recalled from \cite{berglund-hess}.

\begin{defn} Let $X$ be a dualizable chain complex and $C$ a dg coalgebra.  The \emph{canonical coalgebra associated to $X$ and $C$} is   the dg coalgebra $X_*(C)$ with underlying chain complex
$$X_*(C) = X^\vee\tensor C\tensor X,$$
and comultiplication given by the composite 
$$
\xymatrix{X^\vee \tensor C\tensor X \ar[rr]^-{1\tensor \Delta\tensor 1} && X^\vee \tensor C\tensor C\tensor X \ar[d]^-{1\tensor 1\tensor u \tensor 1\tensor 1} \\
&&\big(X^\vee \tensor C\tensor X\big)\tensor \big(X^\vee \tensor C\tensor X\big),}
$$
where $u:\Bbbk \to X\tensor X^{\vee}$ is the coevaluation map.
The \emph{canonical braiding} $(X, T_{C}^{\mathrm{univ}}):C \to X_{*}(C)$ is defined by
$$T_{C}^{\mathrm{univ}}=u\tensor 1 \colon C\tensor  X \rightarrow X\tensor  \big( X^\vee \tensor C\tensor  X \big).$$
The \emph{canonical adjunction associated to $X$ and $C$} is the adjunction governed by the universal braided bimodule $(X,T_C^{\mathrm{univ}})$,
\begin{equation} \label{eq:descent adjunction}
\giantadjunction{\cat {Comod}_{C}}{\cat {Comod}_{X_*(C)}.}{(T_C^{\mathrm{univ}})_*}{(T_C^{\mathrm{univ}})^*}
\end{equation}
We say that \emph{$X$ satisfies effective homotopic descent} if  this adjunction is a Quillen equivalence.
\end{defn}

\begin{rmk}  The canonical braiding determined by a dualizable chain complex $X$ and a dg coalgebra $C$ is universal, in the sense that any braiding $(X,T): C \to C'$ factors as $(X, T_{C}^{\mathrm{univ}}): C\to X_{*}(C)$ followed by the change-of-coalgebras braiding $(\Bbbk, g_{T}): X_{*}(C) \to C'$, where $g_{T}$ is given by the composite
$$X_{*}(C) \xrightarrow {X^{\vee}\otimes T} X^{\vee}\otimes X \otimes C' \xrightarrow {ev \otimes C'} \Bbbk \otimes C'\cong C'.$$
\end{rmk}

\begin{defn} A morphism $g:C \to C'$ of dg coalgebras is \emph{copure} if the counit $g_{*}g^{*}(M) \to M$ of the $g_{*}\dashv g^{*}$ adjunction is a weak equivalence for all fibrant $C'$-comodules $M$.
\end{defn}

\begin{rmk} Since $\Bbbk$ is fibrant (seen as a chain complex concentrated in degree 0) in $\cat {Ch}_{\Bbbk}$, every coalgebra is fibrant as a comodule over itself.  Thus, if $g:C\to C'$ is copure, then $g_{*}g^{*}(C') \to C'$ is a weak equivalence.  Since $g_{*}g^{*}(C')\cong C$, seen as a $C'$-comodule via $g$, it follows that  a copure coalgebra map  is, in particular, a weak equivalence.
\end{rmk}

{The next result is a special case of the second part of \cite [Theorem 4.16]{berglund-hess}.}

\begin{thm}\cite [Theorem 4.16]{berglund-hess}  Let $C, C'$ be in $\cat{Coalg}_{\Bbbk}$.  If $C$ and $C'$ are Morita-Takeuchi equivalent via a braiding $(X,T)$ such that $X$ is dualizable, then $X$ satisfies effective homotopic descent with respect to $C$, and $g_{T}\colon X_*(C) \to C'$ is a copure weak equivalence of corings.
\end{thm}

We can deduce the promised invariance of coHochschild homology from this description of Morita-Takeuchi equivalent coalgebras. 

\begin{prop}\label{prop.MT}  Let $C, C'$ be in $\cat{Coalg}_{\Bbbk}$.  If $C$ and $C'$ are Morita-Takeuchi equivalent via a braiding $(X,T)$ such that the total dimension of $X$ is finite, then  $\cohoch(C)\simeq \cohoch (C')$.
\end{prop}

Note that if $X$ has finite total dimension, then it is certainly dualizable.

\begin{proof} Since $X$ satisfies effective homotopic descent with respect to $C$, the canonical adjunction
\begin{equation*}
\giantadjunction{\cat {Comod}_{C}}{\cat {Comod}_{X_*(C)}.}{(T_C^{\mathrm{univ}})_*}{(T_C^{\mathrm{univ}})^*}
\end{equation*} 
is a dg Quillen equivalence.

It follows that 
$$\cohoch(C)\simeq \hoch(\cat{dgcofree}_{C}) \simeq \hoch (\cat{dgcofree}_{X_{*}(C)})\simeq \cohoch \big(X_{*}(C)\big) \simeq \cohoch (C'),$$
where the first and third weak equivalences follow from Agreement (Proposition \ref{prop:cohoch-props}), and the last equivalence from the fact that $g_{T}$ is copure and therefore a weak equivalence.  

Lemma \ref{lem:q-e}(2) and Lemma \ref{lem:braiding-compatible} suffice to establish the second equivalence, as we now show. To simplify notation, taking $D = X_*(C)$ in Lemma \ref{lem:braiding-compatible}, let $F\dashv G$ denote the adjunction
$$\big(-\otimes_{\Om C}(X\otimes \Om X_{*}(C))\big)\dashv \hom_{\Om X_{*}(C)}(X\otimes \Om X_{*}(C),-),$$
and let $T=T_C^{\mathrm{univ}}$. 

Since $T^*$ is right Quillen, it preserves fibrant objects.  Moreover, if $N$ is a $X_{*}(C)$-comodule such that there exists an $\Om X_{*}(C)$-module $M$ and a quasi-isomorphism $j: N \xrightarrow \simeq R_{X_*(C)}M$, then
$$T^*(j): T^*(N) \xrightarrow \simeq T^*(R_{X_*(C)}M)$$
is also a quasi-isomorphism, since all modules are fibrant, and $T^*$ is a right Quillen functor.  By Lemma \ref{lem:braiding-compatible} 
$$T^*(R_{X_*(C)}M)\cong R_C\circ G(M),$$
so there is a quasi-isomorphism
$$T^*(N) \xrightarrow \simeq R_C\circ G(M)$$
or, equivalently, a quasi-isomorphism
$$L_C T^*(N)\xrightarrow\simeq G(M).$$
Moreover, because 
$$ G(M)=\hom_{\Om X_{*}(C)}(X\otimes \Om X_{*}(C), M)\cong \hom (X,M)\cong X^{\vee}\otimes M,$$
if $M$ is actually an object of $\cat{dgfree}_{\Om X_{*}(C)}$, then $G(M)$ is an object of $\cat {dgfree}_{\Om C}$.
We conclude that $T^{*}$ restricts and corestricts to a functor 
$$T^{*}: \cat {dgcofree}_{X_{*}(C)}\to \cat {dgcofree}_{C}.$$

By Lemma \ref{lem:q-e}(2), to verify that $T^{*}$ is actually a quasi-equivalence, it remains to check that for every $N$ in $\cat {dgcofree}_{C}$, there is an $N'\in \cat{dgcofree}_{X_{*(C)}}$ and a quasi-isomorphism $T_{*}(N)\xrightarrow\simeq N'$.  If $N$ is an object of $\cat {dgcofree}_{C}$, then there is an object $M$ in $\cat{dgfree}_{\Om C}$ and a quasi-isomorphism $j: L_{C}(N) \xrightarrow \simeq M$. Since $F$ is left Quillen, and both $L_{C}(N)$ and $M$ are cofibrant $\Om C$-modules, it follows that $F(j): F\big(L_{C}(N)\big) \to F(M)$ is also a quasi-isomorphism. By Lemma \ref{lem:braiding-compatible}, $F\big(L_{C}(N)\big) \cong L_{X_{*}(C)}\big(T_{*}(N)\big)$, whence there is a quasi-isomorphism $L_{X_{*}(C)}\big(T_{*}(N)\big)\xrightarrow \simeq F(M)$, where $F(M)$ is an object of $\cat {dgfree}_{\Om X_{*}(C)}$.  Indeed, if $M$ is quasi-free on $V$ of finite total degree, then $F(M)$ is quasi-free on $V\otimes X$, which is also free of finite total degree.   If $T_{*}(N)$ is actually fibrant, then it is itself an object of $\cat {dgcofree}_{X_{*}(C)}$, and we can set $N'=T_{*}(N)$.  If not, then for any fibrant replacement of $T_{*}(N)$ in $\cat{Comod}_{X_{*}(C)}$ will be an object of $\cat {dgcofree}_{X_{*}(C)}$ and can play the role of $N'$.
\end{proof}

%%%%%%%%%%%%%%%%%%%%%%%%%%%%%%

\section{Topological coHochschild homology of spectra}\label{sec-spec}
We now consider a spectral version of the constructions and results in section  \ref{sec:dg}. Here we work in any monoidal model category of spectra. 
We show that our results are model invariant in Proposition~\ref{prop.model.invariant} below.

 \subsection{The general theory}
 
Let $k$ be a commutative ring spectrum, $C$  a $k$-coalgebra with comultiplication $\Delta: C \to C \sm_k C$, and $M$  a $C$-bicomodule with right coaction $\rho: M \to M \sm_k C$ and left coaction $\lambda: M \to C \sm_k M$.  Henceforth we write $\sm$ for $\sm_k$ and $C^{\sm n}$ for the $n$-fold smash product of $C$ over $k$.  

\begin{defn}
The {\em coHochschild complex} $\widehat\cH(M,C)$ is the cosimplicial spectrum with
$$  \widehat\cH(M,C)^n = M \sm C^{\sm n} $$
and coface operators  
$$
d^i = \left\{ \begin{array}{ll} \rho \sm \id_C^{\sm n} & i=0 \\   \id_M \sm \id_C^{\sm i-1} \sm \Delta \sm \id_C^{\sm n-i}& 1\leq i \leq n \\ \tau \circ (\lambda \sm \id_C^{\sm n}) & i=n+1  \end{array} \right.
$$
where $\tau: C \sm M \sm C^{\sm n} \to M \sm C^{\sm n+1}$ cycles the first entry to the last entry. {The codegeneracies involve the counit of $C$.}
\end{defn}

Note that one can take $M = C$ with $\lambda = \rho = \Delta $.  In this case $\widehat\cH(C,C) = \widehat\cH(C)$ is the {\em cyclic cobar complex}.

{
Next we define the homotopy invariant notion of {topological coHochschild homology}.  We use $\wb\Tot X^{\bullet}$ to denote the totalization of a Reedy fibrant replacement of the cosimplicial spectrum $X^{\bullet}$.  By~\cite[19.8.7]{hh}, this is a model of the homotopy inverse limit. {\em Topological coHochschild homology} is defined as the derived totalization of the coHochschild complex, 
$$\mbox{coTHH}(M,C) = \wb\Tot \widehat\cH(M,C).$$ 
We abbreviate coTHH$(C,C)$ as coTHH$(C)$.
  
The next statement shows that $\coTHH$ is homotopy invariant.

\begin{lem}\label{lem.invariant}  Let $\cC$ be a monoidal model category of $k$-module spectra. If $f: C \to C'$ is a map of coalgebra spectra in $\cC$ such that  in the underlying category of $k$-module spectra  $f$ is a weak equivalence, and $C, C'$ are cofibrant, then the induced map $\coTHH(C) \to \coTHH(C')$ is a weak equivalence.
\end{lem}

\begin{proof}
Since $C$ and $C'$ are cofibrant, and $\cC$ is a monoidal model category, $f$ induces a levelwise weak equivalence $\widehat\cH(C) \to \widehat\cH(C')$.  Since homotopy inverse limits preserve levelwise weak equivalences, the statement follows. 
\end{proof}

In addition, $\coTHH$ is model independent. 

\begin{prop}\label{prop.model.invariant}
{Topological coHochschild homology} is independent of the model of spectra used.
\end{prop}

\begin{proof}
 This follows from Lemma~\ref{lem.inv} below, since any two monoidal model categories of spectra are connected by Quillen equivalences via a strong monoidal left adjoint.   A universal approach to constructing these monoidal Quillen equivalences is described in~\cite[4.7]{Shi01}; explicit constructions are given in~\cite[0.1,0.2]{MMSS},~\cite[5.1]{schwede-compare}, and~\cite[1.1,1.8]{MM}.  This is also summarized in a large diagram in~\cite[7.1]{moneq}.
 \end{proof}
 
\begin{lem}\label{lem.inv}
Let $L: \cC \to \cD$ be the left adjoint of a strong monoidal Quillen equivalence between two monoidal model categories of $k$-module spectra with $\overline{L}$ the associated derived functor.   Let $C$ be a coalgebra spectrum that is cofibrant as an underlying $k$-module spectrum. 
Then $\mbox{coTHH}(LC)$ is weakly equivalent to $\overline{L}\coTHH(C)$.
\end{lem}

\begin{proof}
Let $R$ be the right adjoint to $L.$
By~\cite[15.4.1]{hh}, levelwise prolongation, denoted $\Ldot, \Rdot,$ induces a Quillen equivalence between the associated Reedy model categories of cosimplicial spectra.
Let $\overline{\Ldot}$ and $\oRdot$ denote the derived functors. 

Since $LC \sm LC \iso L(C\sm C)$, it follows that $LC$ is a coalgebra spectrum in $\cD$ and that  $\Ldot \widehat\cH(C) \iso \widehat\cH(LC)$.  
Since $C$ is cofibrant, $\widehat\cH(C)$ is levelwise cofibrant and so is its cofibrant replacement in the Reedy model structure.  Since $L$ preserves weak equivalences between cofibrant objects, it follows that $\oLdot  \widehat\cH(C)$ is weakly equivalent to $\Ldot\widehat\cH(C)$ and hence also to $\widehat\cH(LC)$.  
  
Applying $\oRdot$ to both sides of this equivalence, we have that $\oRdot \oLdot  \widehat\cH(C)$ is weakly equivalent to $\oRdot \widehat\cH(LC).$ Since $\oLdot$ and $\oRdot$ form an equivalence of homotopy categories, $ \oRdot \oLdot $ is naturally weakly equivalent to the identity and therefore
\begin{equation}\label{eqn:coTHH-L}\widehat \cH(C) \simeq \oRdot \oLdot  \widehat\cH(C) \simeq \oRdot \widehat\cH(LC).\end{equation}

Let $\con X$ denote the constant cosimplicial object on $X$.
Since $\Ldot (\con X) \cong \con (LX)$, the right adjoints also commute, i.e.,  $\lim \Rdot X^{\bullet} \cong R \lim X^{\bullet}$, and so the associated derived functors also commute. In particular, $\oR \mbox{coTHH}(LC)$ is weakly equivalent to the homotopy inverse limit of $\oRdot  \widehat\cH(LC)$.  
It follows from (\ref{eqn:coTHH-L}) that $\coTHH(C) \simeq \oR \mbox{coTHH}(LC)$. Since $\oL$ and $\oR$ form an equivalence of homotopy categories, this is equivalent to the statement in the lemma. 

\end{proof}
}

{

It turns out that coalgebras in spaces with respect to the Cartesian product or in pointed spaces with respect to the smash product are of a very restricted nature. 
The only possible co-unital coalgebra structure on a space is given by the diagonal $\Delta: X \to X \times X$.  Similarly, for a pointed space, the only possible co-unital coalgebra structure exists on a pointed space of the form $X_+$ and is induced by the diagonal $\Delta_{+}: X_+ \to X_+ \sm X_+$.  

It follows that strictly counital coalgebra spectra are also very restricted.
Consider a symmetric spectrum $Z$.  Since the zeroth level of $Z \sm Z$ is the smash product of two copies of level zero of $Z$,  the zeroth level of a co-unital coalgebra symmetric spectrum must have a disjoint base point, which we denote $(Z_{0})_{+}$.   In fact, even more structure is forced in any of the symmetric monoidal categories of spectra.  Let $\Spec$ refer to the $\bS$-modules of \cite{ekmm} or any diagram category of spectra, including symmetric spectra (over simplicial sets or topological spaces, see \cite{HSS, MMSS}), orthogonal spectra (see~\cite{MMSS, MM}), $\Gamma$-spaces (see \cite{Segal, BousfieldFrie}), and $\mathcal{W}$-spaces (see \cite{Anderson}). 

\begin{prop}\label{prop.cocom}~\cite{PS} In $\Spec,$ co-unital coalgebras over the sphere spectrum are cocommutative.  In fact,  if $C$ is a co-unital coalgebra over the sphere spectrum, then $\susp C_{0} \to C$ is surjective.
\end{prop} 

 In an earlier version of this paper, we proved the special case of this proposition for symmetric spectra over simplicial sets. Because of Proposition \ref{prop.cocom}, we focus on suspension spectra in the next section.
}

\subsection{$\coTHH$ of suspension spectra}
{The main statement of this section is the geometric identification of $\coTHH$ of a suspension spectrum as the suspension spectrum of the free loop space, see Theorem~\ref{thm.free.loop}. The proof of this statement is delayed to the following section. This main statement leads to a connection between $\coTHH$ and $\THH$ and another analogue of  ``agreement" in the sense of \cite{mccarthy}.   
Throughout this section by {\em spaces} we mean simplicial sets.}

{\begin{defn}\label{EMSS} For $X$ a Kan complex, consider any model of the loop-path space fibration, $\Omega X \to PX \to X.$
We say that a Kan complex $X$ is an {\em EMSS-good} space if $X$ is connected and $\pi_1 X$ acts nilpotently on $H_i(\Omega X; \bZ)$ for all $i$. 
\end{defn}}

This term refers to the fact that the Eilenberg-Moore spectral sequence for the loop-path space fibration converges strongly by~\cite{dwyer} for any EMMS-good space $X$.
Note that if $X$ is simply connected, then $X$ is certainly EMSS-good.

\begin{thm}\label{thm.free.loop}
If $X$ is an EMSS-good space, then the topological coHochschild homology of its suspension spectrum is equivalent to  the suspension spectrum of the free loop space: 
$$ \coTHH(\susp X)\simeq \susp\cL X.$$
\end{thm}

See also~\cite{Kuhn} and~\cite[2.22]{caryM} for earlier proofs of this statement {for simply connected spaces.}
{The proof of this theorem is given in Section~\ref{sec.cbl} and relies on the proofs of~\cite[4.1, 8.4]{Bousfield} and generalizations discussed in Appendix~\ref{app.A}.  It is likely that this can be further generalized to non-connected spaces $X,$ see for example the proofs of~\cite[3.1, 3.2]{shipley-thesis}.}

For $X$ simply connected, $\THH (\susp \Omega X) \simeq \susp \cL X$ by~\cite{bok-wald}, implying the following corollary.

\begin{cor}\label{cor.thh} Let $X$ be a simply connected Kan complex. There is a weak equivalence between the topological coHochschild homology of the suspension spectrum of $X$ and the topological Hochschild homology of the suspension spectrum of the based loops on $X$:
$$\coTHH(\susp X) \simeq \THH (\susp \Omega X).$$
\end{cor}

{
{
As in the differential graded context, there is also a categorified version of this result.  In~\cite[5.4]{HS16}, somewhat more generally reformulated below in Proposition~\ref{prop.mod}, we show that there is a Quillen equivalence between the categories of module spectra over $\susp \Omega X$ and of comodule spectra over $\susp X$.  Recall from~\cite[5.2]{HS16} that the category of comodules over $\sxp$ admits a model structure, denoted there by $(\comod_{\sxp})_{\pi_*^s}^{\mbox{st}}$ because it is the stabilization of the category of $X_{+}$-comodules with respect to $\pist$-equivalences.   Since this is the only model structure we consider for this category in this paper, we denote it simply by $\comod_{\sxp}$.  Weak equivalences in this structure induce stable equivalences on the underlying spectra by~\cite[5.2 (1)]{HS16}.

The first part of the following result is a simplified version of the statement in \cite[5.4]{HS16}, setting $\cE_* = \pist$. Note that, as above, choosing a base point for $X$ determines a coaugmentation map from the sphere spectrum $\bS \to \sxp$, which in turn determines a $\sxp$-comodule structure on $\bS$.  

\begin{prop}\label{prop.mod}\cite[5.4]{HS16} 
For $X$ a connected space, there is a Quillen equivalence
$$\hugeadjunction{\cat{Mod}_{\soxp}}{\cat{Comod}_{\sxp}}{L}{R}$$
such that $L(\soxp)$ is weakly equivalent to the sphere spectrum as a comodule, and $R(\sxp)$ is weakly equivalent to the sphere spectrum as a module.
\end{prop}

\begin{proof}
The statement of \cite[5.4]{HS16} is formulated for the model of the loop space on $X$ given by the Kan loop group, $\G X$, for $X$ a reduced simplicial set.  

{Here instead we work with $\soxp$, 
where $\Om X$ denotes any model of the loop space of the fibrant replacement of $X$ (i.e., a Kan complex.)} Since $\soxp$ and $\sgx$ have the same homotopy type, their categories of modules are Quillen equivalent.  Moreover any connected simplicial set is weakly equivalent to a reduced simplicial set, 
and replacing $X$ by a weakly equivalent space induces Quillen equivalences on  the respective categories of comodule spectra (see~\cite[5.3]{HS16}). 

The original result is also formulated with respect to a chosen generalized homology theory $\cE_*$, which we fix here to be stable homotopy, $\cE_* = \pi_*^s$.  As in \cite[5.14]{HS16}, since levelwise $\pi_*^s$-equivalences are stable equivalences, one can show that the weak equivalences in this model structure on $\Mod_{\soxp}$  are the stable equivalences on the underlying spectra, i.e., this is the usual model structure on $\Mod_{\soxp}$. (The proof of~\cite[5.14]{HS16} treats the special case where $X$ is a point, but works verbatim for any $X$.)

Concerning the second part of the theorem, 
the left adjoint in~\cite[5.4]{HS16}, $- \sm_{\sgx} \susp \bP X$, takes $\sgx$ to $\susp \bP X$, which is weakly equivalent to $\bS$ since $\bP X$ is contractible. Hence,
 $L(\sgx)\simeq \bS.$ On the other hand, the functor from comodules to modules is the stabilization of the composite of three functors given in~\cite[4.14]{HS16}. By~\cite[3.11]{HS16}, since $X_+$ is the cofree $X_+$-comodule on $S^0$, the first of these functors takes $X_+$ to a retractive space $\Ret_X(S^0)$  over $X$ with total space $S^0 \times X$. The next functor is an equivalence of categories that takes $\Ret_X(S^0)$ to $\Ret_{\bP X}(S^0)$ with a trivial $\G X$-action, which is sent by the third functor to $S^0,$ the trivial, pointed $\G X$-module.  Upon stabilization, this computation implies that the Quillen equivalence on the spectral level sends the comodule $\sxp$ to the module $\bS,$ i.e., $R(\sxp)\simeq \bS.$
\end{proof}
}

As in Proposition~\ref{prop:cohoch-props} in the differential graded context, it follows from Corollary~\ref{cor.thh} and  Proposition~\ref{prop.mod} that topological coHochschild homology for suspension spectra satisfies ``agreement."   Here though, instead of considering finitely generated free modules, we consider the modules that are finitely built from the free module spectrum.  Recall that a subcategory of a triangulated category is called {\em thick} if it is closed under equivalences, triangles, and retracts. Here we also use the same terminology to refer to the underlying subcategory of the model category corresponding to the thick subcategory of the derived category. For example, for $R$ a ring spectrum, we consider $\Thick_R(R)$, the underlying spectral category associated to the thick subcategory generated by $R$.  In the literature, these modules are variously called ``perfect," ``compact," or ``finitely built from $R$."  

Since $L(\soxp) \simeq \bS$, and Quillen equivalences preserve thick subcategories, \cite[5.3, 5.9]{BM} implies the following.

\begin{lem}\label{lem.thh}
The Quillen equivalence in Proposition~\ref{prop.mod} induces a weak equivalence $\THH(\Thick_{\soxp}(\soxp)) \simeq \THH(\Thick_{\sxp}(\bS)).$
\end{lem}

It is a consequence of ~\cite[5.12]{BM} that for any ring spectrum $R$, 
$$\THH(R) \simeq \THH(\Thick_R(R)).$$  The next corollary follows immediately from 
this equivalence for  $R=\soxp$, together with Corollary~\ref{cor.thh} and Lemma~\ref{lem.thh}.

\begin{cor}\label{spec.agree}
Agreement holds for topological coHochschild homology of coalgebra spectra that are suspension spectra.  That is, for any  simply-connected Kan complex $X$, 
 $$\coTHH(\sxp) \simeq \THH (\Thick_{\sxp}(\bS)).$$
 \end{cor}

\begin{rmk}\label{rmk.compact}
Note that $\Thick_{\sxp}(\bS)$ is the subcategory of compact objects in the category of comodules over $\sxp$.  This follows from the Quillen equivalence in Proposition~\ref{prop.mod}, since $\Thick_{\soxp}(\soxp)$ is the subcategory of compact object for modules over $\soxp$.  Since  $R(\sxp)$ is weakly equivalent to $\bS$, where $R$ the right adjoint in Proposition~\ref{prop.mod}, it follows that $\sxp$ is a compact comodule over itself if and only if $\bS$ is a compact module over $\soxp$. In~\cite[5.6(2)]{DGI}, working over $H\bF_p$ instead of $\bS$, it is shown that there are examples where $H\bF_p$ is not compact as a module over 
$H\bF_p \sm \soxp$, e.g., when $X = \bC P^{\infty}$.  
\end{rmk}
}

\subsection{Cobar, Bar, and loop spaces}\label{sec.cbl}
In this section we consider the Cobar and Bar constructions on a suspension spectrum and prove Theorem~\ref{thm.free.loop} from the last section about $\coTHH$ of a suspension spectrum.
The proofs in this section rely on results about the convergence of spectral sequences for cosimplicial spaces that are established in Appendix~\ref{app.A}.

{  
Let $C$ be a $k$-coalgebra spectrum, $N$ a left $C$-comodule with coaction $\lambda: N \to C \sm N$, and $M$ a right $C$-comodule with coaction $\rho: M \to M \sm C$.

\begin{defn}
The {\em cobar complex} $\Omega^{\bullet}(M,C, N)$ is the cosimplicial spectrum with
$$  \Omega(M,C,N)^n = M \sm C^{\sm n}  \sm N$$
with coface operators  
$$
d^i = \left\{ \begin{array}{ll} \rho \sm \id_C^{\sm n} \sm \id_N & i=0 \\   \id_M \sm \id_C^{\sm i-1} \sm \Delta \sm \id_C^{\sm n-i}& 1\leq i \leq n \\ \id_M \sm \id_C^{\sm n} \sm \lambda & i=n+1  \end{array} \right.
$$
The codegeneracies involve the counit of $C$.
\end{defn}

If $C$ is a coaugmented $k$-coalgebra with coaugmentation $\eta: k \to C$, i.e., $\eta$ is a homomorphism of coalgebras such that $\epsilon \eta = \id_C$, then $\eta$ endows $k$ with the structure of a  $C$-bicomodule. In this case $\Omega(k,C,k) = \Omega^{\bullet}(C)$ is the {\em cobar complex of $C$.}  Its derived totalization is the {\em cobar construction on $C$}:
$$\Cobar (C) = \wb \Tot \Omega^{\bullet}(C).$$

}

The following cosimplicial resolution of the mapping space plays an important role in the statements below.  

\begin{defn}\label{defn:mapping-space} Let $W$ and $Z$ be pointed simplicial sets with $Z$ a Kan complex, and 
let $\map_*(W_{\bullet},Z)$ be the cosimplicial space with $\map_*(W,Z)^n$ equal to a product of copies of $Z$ indexed by the non-base point $n$-simplices in $W$, with cofaces and codegeneracies induced by those of $W$.  The pointed mapping space $Z^W$ agrees with the totalization of this cosimplicial space.  
\end{defn}

 If $W = S^1 = \Delta[1] / \del \Delta[1]$, then $\map_*(S^1_{\bullet}, Z)^n = Z^{\times n}$ for all $n$, and the totalization is $\Omega Z$, a simplicial model for the based loop space on $|Z|$.

For $X$ a pointed space, there is a canonical map $S^0 \to X$ that gives rise to a coaugmentation $\bS \to \sxp.$ Thus we can consider the cobar construction on $\sxp.$

\begin{prop}\label{prop-loops} If $X$ is pointed and an EMSS-good space,  then the cobar construction on the suspension spectrum of $X$ is weakly equivalent to the suspension spectrum of the pointed loops on $X$:
$$ \Cobar (\sxp) \simeq \soxp.$$
\end{prop}

\begin{proof}
If we add a disjoint base point, then $\map_*(S^1_{\bullet}, X)_+$ has cosimplicial level $n$ given by $(X^{\times n})_+ \iso (X_+)^{\sm n}$. Applying the suspension spectrum functor, we see that $\Sigma^{\infty} \map_*(S^1_{\bullet}, X)_+$ agrees with the cobar complex $\Omega^{\bullet} (\sxp) $.  

{
By~\cite{dwyer}, if  $X$ is EMSS-good, the Eilenberg-Moore spectral sequence converges for ordinary homology with integral coefficients. This Eilenberg-Moore spectral sequence is the homology spectral sequence for the cosimplicial space $\map_*(S^1_{\bullet}, X).$ By Corollary~\ref{cor-conv} the strong convergence of this spectral sequence implies that the total complex commutes with the suspension spectrum functor, i.e., 
$$\oT \Sigma^{\infty} \map_*(S^1_{\bullet}, X)\simeq \Sigma^{\infty} \oT \map_*(S^1_{\bullet}, X).$$ 
By Proposition~\ref{prop.agree}, we can add disjoint base points to this equivalence, obtaining that 
$$\oT \Sigma^{\infty} \map_*(S^1_{\bullet}, X)_+\simeq \Sigma^{\infty} \oT \map_*(S^1_{\bullet}, X)_+.$$ 
Since $\oT \map_*(S^1_{\bullet}, X)_+ \simeq \Omega X_+$, we can conclude.}  
\end{proof}

{A dual to Proposition \ref{prop-loops}, with a considerably simpler proof, holds as well.  The statement of the dual is formulated in terms of the \emph{Kan classifying space functor}, $\overline W\colon \cat{sGp} \to \cat {sSet}_{0}$, from simplicial groups to reduced simplicial sets.  A detailed definition of this functor can be found in \cite{cegarra-remedios}, where it is also shown that $\overline W$ factors as the composite $\operatorname{codiag} \circ N$, where $N\colon \cat{sGp} \to \cat {ssSet}$ is the levelwise nerve functor from simplicial groups to bisimplicial sets, and $\operatorname{codiag}: \cat {ssSet} \to \cat{sSet}$ is the \emph{Artin-Mazur codiagonalization functor} \cite{artin-mazur}.  The functor $\operatorname{codiag}$ is often called \emph{Artin-Mazur totalization} and denoted $\Tot$, which we avoid, due to the risk of confusion with the other notion of totalization that we employ in this article.   As we do not make any computations based on the explicit and somewhat involved formula for $\operatorname{codiag}$, we do not recall it here.

We also consider the \emph{bar construction functor}, denoted $\bBar$, which associates to any associative ring spectrum $R$ a spectrum $\bBar R=|B_{\bullet }R|$, where $|-|$ denotes geometric realization, and $\bBar_{\bullet}R$ is the simplicial spectrum with $\bBar_{n}R=R^{\wedge n}$, face maps built from the multiplication map of $R$, and degeneracies from its unit map.

\begin{prop}\label{prop-bar} For any simplicial group  $G$, the bar construction on the ring spectrum $\susp G$ is naturally weak equivalent to the suspension spectrum of the bar construction of $G$, i.e., 
$$ \bBar (\Sigma^{\infty} _+ G) \simeq \Sigma^{\infty}_+\overline{W} G.$$
\end{prop}

\begin{proof}  Cegarra and Remedios proved in \cite{cegarra-remedios} that the obvious natural transformation from the diagonalization functor $\operatorname{diag}: \cat {ssSet} \to \cat{sSet}$ to $\operatorname{codiag}$ is in fact a natural weak equivalence.  It follows that for any simplicial group $G$ there is a sequence of natural weak equivalences and isomorphisms
$$\Sigma^{\infty}_+\overline{W} G \simeq \Sigma^{\infty}_+ \operatorname{diag} NG\cong |\Sigma^{\infty}_+ NG| \cong | \bBar_{\bullet} \Sigma^{\infty}_+ G|= \bBar (\Sigma^{\infty} _+ G),$$
where straightforward computations suffice to establish the two isomorphisms.
\end{proof}
}

The next lemma is the first step in the proof of Theorem~\ref{thm.free.loop}. Note that we consider unpointed mapping spaces here.

\begin{lem}\label{lem.agree} For any space $X$, there is an isomorphism of cosimplicial spectra
$$\coTHH^{\bullet}(\sxp)\cong \susp \map(S^1_{\bullet},X).$$
\end{lem}

\begin{proof} Since these are both cosimplicial suspension spectra, it is enough to establish the isomorphism on the 0th space level. The $0$th space of $\coTHH(\sxp)$ has $n$th cosimplicial level  $(X_+)^{\sm (n + 1)} \cong (X^{\times (n + 1)})_+$, which agrees with $\map(S^1_{\bullet}, X)^n_+$. In both cases, the coface maps are induced by diagonals on the appropriate factor (with one extra twist for $d^{n+1}$), while  the codegeneracy maps are projections onto the appropriate factors.
\end{proof}

\begin{proof}[Proof of Theorem~\ref{thm.free.loop}]
Proposition~\ref{prop.agree}  implies that it is sufficient to prove the statement with disjoint base points removed, so
 it suffices to show that $$\oT \Sigma^{\infty} \map(S^1_{\bullet}, X) \simeq \Sigma^{\infty} \oT \map(S^1_{\bullet}, X).$$ 
 By Corollary~\ref{cor-conv}, it is enough to know that the Anderson spectral sequence for homology with coefficients in $\bZ$  for the cosimplicial space $ \map(S^1_{\bullet}, X)$ strongly converges. By Proposition~\ref{prop.anderson} this holds for $X$ an EMSS-good space, as required in the hypotheses here.
 \end{proof}

\appendix
\section{Total complexes of cosimplicial suspension spectra}\label{app.A}
In this section we prove several useful results concerning cosimplicial spectra, their associated spectral sequences, and commuting certain homotopy limits and colimits.  The most general statement,
Proposition~\ref{prop-conv-D}, gives conditions in terms of convergence of the associated spectral sequence for commuting the derived total complex (a homotopy limit) with smashing with a spectrum (a homotopy colimit). In this paper, we need only the suspension spectrum case, stated in Corollary~\ref{cor-conv}.  The convergence conditions in the hypothesis here are verified in Proposition~\ref{prop.anderson} for the Anderson spectral sequence for the cosimplicial space $ \map(S^1_{\bullet}, X)$. These statements are then used in the proofs of Theorem~\ref{thm.free.loop} and Proposition~\ref{prop-loops} above.  Proposition~\ref{prop.agree} shows that a variation of Corollary~\ref{cor-conv} holds even after adding base points.

\begin{prop}\label{prop-conv-D}
If the spectral sequence associated to the cosimplicial space $Y^{\bullet}$ for the generalized homology theory $D_*$  converges strongly, then $$\oT (D \sm Y^{\bullet}) \simeq D \sm \oT Y^{\bullet}.$$  
\end{prop}

Bousfield, in~\cite{Bousfield}, shows that the following conditions imply strong convergence for such spectral sequences.

\begin{prop}\cite[3.1]{Bousfield}\label{prop-Bousfield}  Let $R$ be a ring such that $R \subset \bQ$ or $R = \bZ / p$ for $p$ a prime.  If $Y^{\bullet}$  is a cosimplicial space such that the associated homology spectral sequence  with coefficients in $R$ strongly converges to $H_*(\oT Y^{\bullet}; R)$, then for each connective spectrum $D$ with $R$-nilpotent coefficient groups $\pi_i D$, the spectral sequence associated to $Y^\bullet$ for the generalized homology theory $D_*$ converges strongly to $D_*(\oT Y^{\bullet})$.
\end{prop}

Since abelian groups are $\bZ$-nilpotent, the following corollary of Propositions~\ref{prop-conv-D} and~\ref{prop-Bousfield} holds. 

\begin{cor}\label{cor-conv}
If the integral  spectral sequence for the cosimplicial space $Y^{\bullet}$  strongly converges, then
$$\oT (\Sigma^{\infty}Y^{\bullet}) \simeq \Sigma^{\infty} \oT Y^{\bullet}.$$  
\end{cor}

{
\begin{proof}[Proof of Proposition~\ref{prop-conv-D}]
Recall {from~\cite[X.6.1]{bousfield.kan}} that there is a homotopy spectral sequence for any cosimplicial space $Y^{\bullet}$ that converges to the homotopy of $\oT Y^{\bullet}$ under mild conditions.  This spectral sequence arises from the tower of fibrations given by $\{\Tot^s(Y^{\bullet})\}$ and has $E_2$-term given by $\pi^s\pi_t Y^{\bullet}.$

Rector's spectral sequence for computing the $D_*$-homology of a cosimplicial space is considered in~\cite[2.4]{Bousfield}, where it is constructed as the homotopy spectral sequence for the cosimplicial spectrum given by $D \sm Y^{\bullet}$.  The $E_2$-term is therefore given by $\pi^s\pi_t (D \sm Y^{\bullet}) \cong \pi^s D_t(Y^{\bullet}),$ and it abuts to $\pi_* \oT (D \sm Y^{\bullet}).$  By~\cite[2.5]{Bousfield}, strong convergence for this spectral sequence implies that $D_*(\oT Y^{\bullet})$ is isomorphic to $\pi_* \oT (D \sm Y^{\bullet}).$   Strong convergence for the homology spectral sequence for $D_*$ thus implies the statement in the proposition.  
\end{proof}
}

Our next goal is to prove the following strong convergence result.

\begin{prop}\label{prop.anderson} For $X$ an EMSS-good space,
the Anderson spectral sequence for homology with coefficients in $\bZ$  for the cosimplicial space $ \map(S^1_{\bullet}, X)$ strongly converges to $H_*(\cL X; \bZ)$.  
 \end{prop}
 
This strengthens the convergence results in~\cite[4.2]{Bousfield} that require $X$ to be simply connected.  We expect that  Proposition~\ref{prop-pullback} should enable similar generalizations for other mapping spaces.

To prove Proposition \ref{prop.anderson}, we need the following definitions and result.
A cosimplicial space is {\em R-strongly convergent} if the associated homology spectral sequence with coefficients in $R$ strongly converges. If $R= \bZ$ we often leave off the $R$.  It is \emph{$R$-pro-convergent} if the homology spectral sequence with coefficients in $R$ converges to the associated tower of partial total spaces; see~\cite[8.4]{Bousfield} for details. In each application in this paper, the associated tower of partial total spaces is eventually constant, so pro-convergence is equivalent to strong convergence in cases relevant to us. The next result, from~\cite[8.4]{Bousfield} and generalized to non-contractible $Y^{\bullet}$ in~\cite[3.2]{shipley-thesis}, is formulated in terms of  $R$-pro-convergent cosimplicial spaces.

Consider a pull-back square of cosimplicial spaces 
$$\xymatrix{  M^{\bullet} \ar[rr] \ar[d] && Y^{\bullet} \ar[d]^{f}\\
X^{\bullet}  \ar[rr] && B^{\bullet}.} 
$$
There are associated pull-back squares for each cosimplicial level $n$
 
 $$\xymatrix{  M^n \ar[rr] \ar[d] && Y^n \ar[d]\\
X^n \ar[rr] && B^n} 
$$
  and for each partial total space $\Tot_s$ 
    $$\xymatrix{  \Tot_s M \ar[rr] \ar[d] && \Tot_s Y \ar[d]\\
\Tot_s X  \ar[rr] && \Tot_s B.} 
$$

\begin{prop}\label{prop-pullback}
~\cite[8.4]{Bousfield}, ~\cite[3.2]{shipley-thesis} Consider a pull-back square of cosimplicial spaces as above, with $f$ a fibration and $X^{\bullet}$, $Y^{\bullet}$, and $B^{\bullet}$ fibrant.  If $X^{\bullet}$, $Y^{\bullet}$, and $B^{\bullet}$ are $R$ pro-convergent, and the Eilenberg-Moore spectral sequences for the pull-back squares above for each cosimplicial level $n$ and each total level $s$ strongly converge, then  $M^{\bullet}$ is $R$-pro-convergent 
\end{prop}

\begin{proof}[Proof of Proposition~\ref{prop.anderson}]
Since $S^1$ is the following pushout in simplicial sets
   $$\xymatrix{   \text{*} \ar[d]\ar[rr] && \Delta[1] \ar[d] \\
S^0 \ar[rr] && S^1,} 
$$
the cosimplicial space $ \map(S^1_{\bullet}, X)$ is a pull-back:
  $$\xymatrix{  \map(S^1_{\bullet}, X) \ar[d] \ar[rr] && \map(\Delta[1]_{\bullet}, X) \ar[d] \\
    \map(S^0_{\bullet}, X). \ar[rr] && \map(\text{*}_{\bullet}, X).} 
$$
We use Proposition~\ref{prop-pullback} to establish strong convergence of the homology spectral sequence with coefficients in $\bZ$ for $ \map(S^1_{\bullet}, X)$ by verifying the hypotheses listed there.  

Since the inclusion $\text{*} \to \Delta[1]$ is a cofibration of simplicial sets, and $X$ is fibrant, the righthand vertical map above is a fibration of cosimplicial spaces. Also because $X$ is fibrant, the four corners are fibrant cosimplicial spaces. 

The cosimplicial spaces in the two bottom corners are constant, with each level given by $X$ and $X^2$ respectively.  Hence the associated homology spectral sequences are strongly convergent.   The cosimplicial space for the top right corner is equivalent to the cosimplicial space given by Rector's geometric cobar construction for the pullback of the identity maps:
  $$\xymatrix{ && X  \ar[d]\\
X  \ar[rr] && X.} 
$$
By~\cite[4.1]{Bousfield}, since the fibers here are trivial, the associated homology spectral sequence converges as long as $X$ is connected.

Next we consider the Eilenberg-Moore spectral sequences associated to the pullbacks in each level.
In level $n$ the pullback is given by
   $$\xymatrix{   && X^{n+2}  \ar[d]^{\pi_{0,n+2}}\\
X  \ar[rr]^{\Delta} && X^2.} 
$$ 
Again, we apply~\cite[4.1]{Bousfield}.  Here the vertical map is projection onto the first and last factors, so the action of $\pi_1(X^2)$ on the homology of the fiber is trivial. Thus, the associated spectral sequence strongly converges.
  
Finally, we consider the Eilenberg-Moore spectral sequences associated to the pullbacks of partial total spaces.  $\Tot_0$ agrees with cosimplicial level zero, so it is covered above for $n=0.$
  By~\cite[X.3.3]{bousfield.kan}, $\Tot_s \map(Z_{\bullet}, Y) \cong \map (Z_{\bullet}^{[s]}, Y)$ where $Z_{\bullet}^{[s]}$ is the $s$-skeleton of $Z$.
 It follows that $\Tot \cong \Tot_s$ for all $s$ in the bottom two corners.
Also, since $\Delta[1]$ is one dimensional,  $\Tot_s\map(\Delta[1]_\bullet, X) \cong \map(\Delta[1]_\bullet, X) $ for $s\geq 1.$  
 So we have the same pullback of partial total spaces for each $s \geq 1$,
   $$\xymatrix{   && \map(\Delta[1], X) \ar[d]\\
X  \ar[rr]^{\Delta} && X^2.} 
$$
Choose a point in $X^2$ in the image of the diagonal map $\Delta$, so that the fiber over that point is the pointed loop space $\Omega X$.  Since $X$ is EMSS-good, the action of the fundamental group $\pi_1(X)$ on the homology $H_*(\Omega X)$ for the path loop space fibration is nilpotent. It follows that for the vertical fibration above, $\pi_1 (X^2)$ also acts nilpotently on $H_*(\Omega X)$. 
Thus, by ~\cite[4.1]{Bousfield}, the associated spectral sequence strongly converges. 
\end{proof}

The following proposition considers the effect on the weak equivalence of Corollary \ref{cor-conv}  of adding disjoint base points.

{
\begin{prop}\label{prop.agree} If $Y^{\bullet}$ is a cosimplicial space such that $\oT \Sigma^{\infty} Y^{\bullet}$ is weakly equivalent to $\Sigma^{\infty} \oT Y^{\bullet}$, then $\oT \susp Y^{\bullet}$ is weakly equivalent to $\susp (\oT Y^{\bullet}).$
\end{prop}

\begin{proof}  To be definite, we work here in the underlying model category of symmetric spectra of simplicial sets~\cite{HSS}, where the suspension spectrum $\Sigma^{\infty} X$ is always cofibrant.  Thus we do not need to derive the coproducts below, but we do need to consider the derived product, denoted $\overline{\times}$. 

 Recall that in the homotopy category of spectra, the coproduct is equivalent to the product, i.e.,
 $$W\vee Z \simeq W\ \overline\times \ Z$$
 for all spectra $W$ and $Z$. 

Let $\bbS \cong \Sigma^{\infty} S^0$ denote the sphere spectrum.  If $f\Sigma^{\infty} X$ is a fibrant replacement of $\Sigma^{\infty} X$ for some space $X$, then 
$$\susp X \cong \Sigma^{\infty} X \vee \bbS\simeq \Sigma^{\infty} X \ \overline{\times}\  \bbS \simeq
f\Sigma^{\infty} X \times f\bbS.$$

Let $f\susp Y^{\bullet}$ denote the fibrant replacement of $\susp Y^{\bullet}$ in the Reedy model category of cosimplicial spectra.
By the argument above, $f\susp Y^{\bullet}$ is levelwise weakly equivalent to $f\Sigma^{\infty} (Y^{\bullet}) \times c^{\bullet}f\bbS$, where $c^{\bullet} $ denotes the constant cosimplicial spectrum functor.    Since totalization commutes with products, 
$$\oT \susp Y^{\bullet}\simeq \oT \Sigma^{\infty} (Y^{\bullet}) \times \oT c^{\bullet}f\bbS\simeq \oT\Sigma^{\infty} (Y^{\bullet}) \times f\bbS.$$   
The hypothesis of the proposition, together with the weak equivalence between products and coproducts, implies that
$$\oT\Sigma^{\infty} (Y^{\bullet}) \times f\bbS \simeq \Sigma^{\infty} \oT Y^{\bullet} \vee f\bbS.$$ 
Since $\bbS \to f\bbS$ is a trivial cofibration, this last term is weakly equivalent to $\susp (\oT Y^{\bullet})$, as desired.
\end{proof}
}

\nocite{*}
\renewcommand{\bibname}{References}
\bibliographystyle{plain.bst}

\end{document}